\documentclass[12pt]{article}
\usepackage{amsmath}
\usepackage{amssymb}

\usepackage[dvips]{graphics}
\usepackage[dvips]{psfrag}
\usepackage[final]{showkeys}

\usepackage{hyperref}
\newtheorem{theorem}{Theorem}[section]
\newtheorem{definition}[theorem]{Definition}
\newtheorem{lemma}[theorem]{Lemma}
\newtheorem{corollary}[theorem]{Corollary}

\newenvironment{proof}
{\noindent
{\bf Proof.}}
{\hfill $\square$\medskip

}

\renewcommand{\(}{\begin{equation}}
\renewcommand{\)}{\end{equation}}

\newcommand{\id}{\operatorname{Id}}

\newcommand{\R}{\mathbb{R}}
\newcommand{\Q}{\mathbb{Q}}

\newcommand{\LL}{\mathbb{L}}
\newcommand{\N}{\mathbb{N}}
\newcommand{\Z}{\mathbb{Z}}
\newcommand{\E}{\mathbb E}

\newcommand{\calS}{{\mathcal S}}

\newcommand{\calL}{{\mathcal L}}

\newcommand{\bG}{{\bf G}}
\newcommand{\bg}{{\bf g}}
\newcommand{\bP}{{\bf P}}
\newcommand{\bbP}{{\bf \bar P}}

\newcommand{\bp}{{\bf p}}
\newcommand{\bq}{{\bf q}}
\newcommand{\bh}{{\bf h}}

\newcommand{\bd}{{\bf d}}

\newcommand{\bR}{{\bf R}}

\newcommand{\bA}{{\bf A}}
\newcommand{\bbA}{{\bf \bar A}}
\newcommand{\ba}{{\bf a}}
\newcommand{\bB}{{\bf B}}
\newcommand{\bE}{{\bf E}}
\newcommand{\be}{{\bf e}}

\newcommand{\bgl}{\bg_\lambda}
\newcommand{\bgL}{\bg_\Lambda}
\newcommand{\bgS}{\bg_\Sigma}

\begin{document}

\title{H{\"o}lder Regularity of Geometric Subdivision Schemes}

\author{T.~Ewald \and U.~Reif \and M.~Sabin}

\date{\today}

\maketitle

\begin{abstract}
We present a framework for analyzing non-linear $\R^d$-valued subdivision schemes which are
geometric in the sense that they commute with similarities in $\R^d$.
It admits to establish $C^{1,\alpha}$-regularity for arbitrary schemes of this type,
and $C^{2,\alpha}$-regularity for an important subset thereof, which includes all
real-valued schemes.
Our results are constructive in the sense that they can be verified 
explicitly for any scheme and any given set of initial data
by a universal procedure. This procedure can be executed automatically and rigorously
by a computer when using interval arithmetics.
\end{abstract}



\section{Introduction}
Univariate subdivision schemes define a function or curve as the limit of a 
refinement process, starting from a 
sequence of point positions usually called the `control polygon'.   
Since the first introduction of such schemes a question of great interest has been 
`Under what conditions will the limit curve exist, and what H\"older continuity will it have?'.  
Techniques have been honed since the first papers \cite{StandardDyn91,StandardCavaretta91}, 
and there is now a standard approach 
which can be applied to any newly proposed subdivision scheme which is linear 
(the new points are defined by specific linear combinations of old ones), uniform 
(the linear combination coefficients are the same at each part of the sequence) and 
stationary (the coefficients are the same at all steps), see \cite{StandardDyn02} for
a survey or \cite{Sabin:2010} for a comprehensive exposition. 
While linear subdivision is meanwhile well understood,
non-linear algorithms have gained some interest in recent years. 
The variety of different schemes may be grouped as follows:

First, there are 
{\em manifold-valued schemes} where non-linearity comes from adapting
a linear scheme to the special structure of the space carrying the data. 
Today, this class of schemes is fairly well understood, see, for instance,
\cite{ManWallner05,ManDonoho2005,ManWallner06,NonlinYu07,ManGrohs08,ManGrohs10,ManWallner11}.
Using the concept of proximity, it is shown that, roughly speaking, regularity of
the linear scheme is inherited by the non-linear scheme.

Second, there are non-linear {\em real-valued schemes} where standard linear averaging rules
are replaced by more general procedures. 
For instance, in \cite{Goldman2008}, arithmetic means
are replaced by geometric means. Other examples include schemes based on
median-interpolating polynomials \cite{Donoho2000,Oswald2004,NonlinYu05}, 
interpolating rational functions \cite{Kuijt99}, or
interpolating circles \cite{BspFloater13}. 
For all these schemes, some specialized smoothness analysis is
available. More general arguments can be found in
\cite{NonlinDaubechies04,QuasiOswald03,QuasiDyn03}. Still, the verification of the 
conditions given there seems to be rather intricate in a specific setting. 

Third, and this is the class of algorithms that inspired this paper, there are
{\em geometric subdivision schemes} for generating planar or spatial curves with rules motivated by 
some geometric considerations.
The first paper in this direction is probably \cite{DeBoor:1987}, where it is shown that
`cutting corners always works'. Later on, in \cite{BspSabin05}, a variant of the four-point
scheme is suggested for curve design. Here, locally interpolating polynomials defining the linear
four-point scheme are replaced by
interpolating circles to obtain a circle-preserving scheme. 
In the same spirit, a geometric modification of the
Lane-Riesenfeld algorithm is developed in \cite{BspCashman2013}.
Another non-linear variant of the four-point scheme can be found in \cite{BspFloater09}, where
the parametrization underlying the local interpolation is adapted to account for an uneven spacing of control points.

In this paper, we develop a general framework for the analysis of geometric subdivision
schemes. Essentially, these schemes are characterized by the fact that they commute with
similarity transformations. In particular, the schemes in \cite{BspSabin05,BspCashman2013,BspFloater09}
are covered, but also some real-valued schemes, as monotonicity preserving subdivision \cite{Kuijt99}
or median-interpolating subdivision \cite{Donoho2000}  and its generalizations \cite{NonlinYu05}. 
It is not applicable to manifold-valued schemes or to Goldman's algorithm \cite{Goldman2008} because
these schemes are not invariant with respect to similarities.
Unlike the truly geometric $G^1$-analysis in \cite{Dyn12}, which relies on the decay of angles in the sequence 
of control polygons, our approach is parametric. That is, it aims at establishing H{\"o}lder continuity of 
a special parametrization of the limit curve,
which is in some sense uniform. For the schemes to be considered this means that linear sequences
of control points have to be mapped to linear sequences with half spacing. This rules out
a treatment of de~Rham's scheme with variable cutting ratio \cite{deRham:1956}, except for the special case of
Chaikin's algorithm \cite{Chaikin:1974}. 
Still, we claim that our approach is fairly general and covers a broad
class of algorithms in a systematic and constructive way.

Anticipating subsequent denotation, control polygons and control points will now be referred
to as {\em chains} and {\em points}, respectively.
While standard analysis relies on the asymptotics of certain higher order differences, our
approach is based on studying the rate of decay of {\em relative distortion}. This
quantity measures the deviation of groups of points from linear behavior and
is invariant with respect to similarities. We will show how to establish parameters $\alpha$ and $\delta$
with the following property: if the relative distortion falls below $\delta$, then
it decays at some rate towards zero which guarantees that the limit curve is at least $C^{1,\alpha}$.
Once $\delta$ is known, any given chain can be checked for compliance. If so, the
limit curve is known to be $C^{1,\alpha}$. Otherwise, a few subdivision iterates can be computed, and
then these followers may be checked again. Thus, first order H{\"o}lder regularity
can be established by a sequence of steps which are, at least in principle, implementable as a computer
program. Rigor can be guaranteed when using interval arithmetics.
Proximity to some linear scheme, as crucial for the analysis of manifold valued subdivision,
is also used for our analysis, but in a rather unspecific way. For instance, we may employ 
cubic B-spline subdivision as linear reference for the $C^{1,\alpha}$-analysis of any primal scheme.
A subdivision scheme will be called {\em locally linear} if
its derivative evaluated at a linear chain
gives rise to a linear subdivision scheme, called its {\em linear companion}. 
For such schemes, we are able to show that H{\"o}lder regularity up to second order
is inherited from the linear companion. This result applies in particular to any
real-valued algorithm within the class of GLUE-schemes, as defined below.

The paper is organized as follows:
In Section~2 we introduce some basic notation used in the body of the paper.   
Also the concept of {\em relative distortion} is defined as an analog to second 
differences in linear standard theory.

In Section~3 we define a class of geometric subdivision schemes, called GLUE-schemes,
to which our subsequent analysis applies. The four letters of the acronym address
schemes being 
{\bf g}eometric (i.e., commute with similarities), 
{\bf l}ocal (i.e., new points depend only on a fixed number of old ones), 
{\bf u}niform (i.e., the same rules apply everywhere), and 
{\bf e}quilinear (i.e., linear chains are mapped to linear chains with half spacing). 
Here, we limit ourselves to binary schemes, 
but the generalization to arbitrary arity is a straightforward one, which the reader is 
encouraged to carry out.

In Section~4 the concept of {\em straightening} is introduced. It is used to
quantify the decay of relative distortion as subdivision proceeds.
In particular, we establish tools for determining a neighborhood 
of linear chains in which straightening at a certain geometric rate can be
guaranteed.

In  Section~5 the three different manifestations of straightening are related
to continuity, differentiability, and first order H{\"o}lder regularity
of limit curves with respect to a natural uniform parametrization. 

In Section~6 we enhance our results by examining the derivative of the
subdivision map at the standard linear chain. In general, this derivative
corresponds to two linear subdivision schemes, related in some sense 
to the behavior of the GLUE-scheme in tangential and normal direction. 
We show that the worse of these two schemes determines its first order H{\"o}lder 
regularity. If, as for circle-preserving subdivision \cite{BspSabin05}, the
two linear schemes coincide or if the scheme is real-valued, 
this single linear scheme bequeaths even {\em second order} H{\"o}lder regularity.

As an example, we consider the scheme which motivated this analysis, 
the {\em circle preserving subdivision (CPS)} introduced in \cite{BspSabin05}.   
This scheme is a variant of the famous {\em four-point scheme (FPS)} due
to Dubuc \cite{StandardDubuc86}. Usually, FPS is explained by evaluating
an interpolating cubic at its parametric mid-point. However, also the
following three-step procedure gives an accurate description for
the computation of a new point:
First, two interpolating quadratics are determined for the three
leftmost and the three rightmost points out of four consecutive ones.
Second, the second devided differences of the two quadratics are averaged. 
Third, the quadratic interpolating the second and the third point with
the averaged second devided difference is evaluated at the midpoint
to obtain the new point.

This is modified for CPS in the following way:
First, interpolation by quadratics becomes interpolation by circles.
Second, second devided differences of the quadratics is replaced by the curvatures
of circles. Third, the interpolating quadratic is replaced by an interolating circle
with averaged curvature, and 
evaluation at the modpoint is replaced by picking a point on the circle with a
certain distance ratio to the two old neighboring points, see Figure~\ref{fig}.

\begin{figure}
\[
\includegraphics{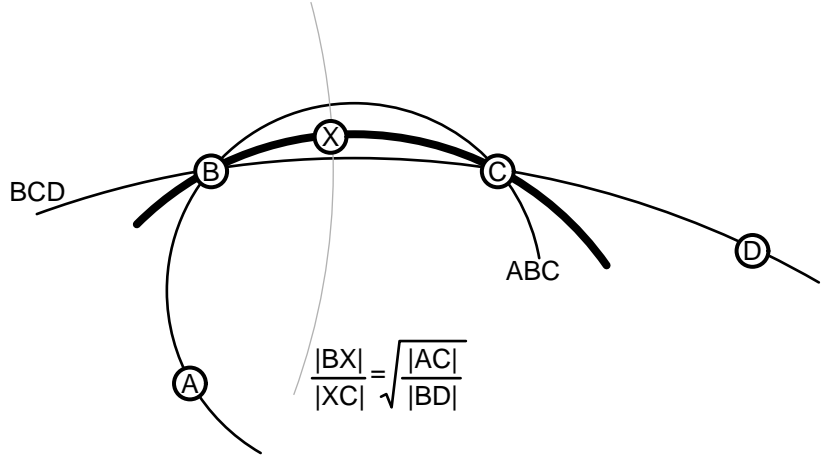}
\]
\caption{$A$, $B$, $C$ and $D$ are four consecutive original points.  
The new point $X$ lies on the bold circle, which has curvature intermediate between the 
circles $ABC$ and $BCD$.  This ensures that the scheme preserves circles. It also lies 
on the locus where $|BX|/|CX| = \sqrt{|AC|/|BD|}$, which ensures that the spacing between 
points tends to uniform locally in the limit.}
\label{fig}
\end{figure}

\section{Setup}
\label{sec:setup}
Even though our analysis is motivated by subdivision schemes generating
curves in two- or three-space, it is equally valid for arbitrary 
space dimensions $d \in \N$, including the real-valued case $d=1$.
Fixing $d$,
we denote Euclidean $d$-space by $\E := \R^d$
and investigate subdivision algorithms acting
on sequences of points in $\E$, called {\em chains}. 
The space of chains with $N \in \N$ points is denoted by $\E^N$.
Chains with at least $n$ points form the set $\E^{\ge n} := \bigcup_{N \ge n} \E^N$.
Joining the points of a chain by straight line segments
yields a polyline, which is also called the {\em control polygon}
of the curve to be generated. Because the polyline structure is
not relevant for our analysis, we prefer to talk about chains.

Points are understood as row vectors, implying that linear
maps are represented by matrices multiplying from the right hand side.
Columns of points yield chains. 
Chains are denoted by upper case bold face letters,
and the corresponding standard lower case letters, tagged with a subscript,
are used for the points. That is,
\[
  \bP = [p_0;\dots;p_{N-1}]\in \E^N
	,\quad
	p_i = [p_{i,0},\dots,p_{i,d-1}] \in \E
\]
Here and below, 
when specifying vectors and matrices, a comma is separating columns, 
while a semicolon is separating rows (as in MATLAB). 
Indices in vectors and matrices always start from $0$.

The {\em length} of the chain $\bP = [p_0;\dots;p_{N-1}] \in \E^{\ge n}$
is the number of its points and denoted by $\# \bP := N$. 
While the length of chains is increasing as subdivision proceeds,
analysis can be based on the study of subchains of a certain fixed
length $n$, called the {\em spread} of the algorithm under consideration.
For $\bP \in \E^{\ge n}$, such subchains are addressed by means of
truncation operators, 
\[
	T^n_i \bP := [p_{i};\dots;p_{i+n-1}] \in \E^n
	,\quad
	0 \le i \le \#\bP-n
	.
\]
The set of all subchains is $T^n \bP := \{T^n_i \bp : 0 \le i \le \#\bP-n \}$.

With $e := [1,0,\dots,0] \in \E$ the first unit vector,
let 
\[
  \bE_N := [e;2e;\dots;Ne]
  ,\quad 
  \be := [e;2e;\dots;ne]
\]
denote the {\em standard linear chains} in $\E^N$ and $\E^n$, respectively. 
Here and throughout, lower case bold face letters are reserved for chains in $\E^n$.
Also, as $n$ is a fixed parameter, we take the liberty of omitting it occasionally 
in the notation of functions or other objects depending on it.

The {\em forward difference operator $\Delta : \E^{\ge n} \to \E^{\ge(n-1)}$} is defined by 
$(\Delta \bP)_i := p_{i+1}-p_i$.
Repeated application yields the chain $\Delta^k \bP \in \E^{\ge(n-k)}$ of differences of
order $k < n$. The points forming this chain are 
\[
  \Delta^k p_i := (\Delta^k \bP)_i
  ,\quad
  i = 0,\dots,\#\bP-k-1
  .
\]
The Euclidean norm on $\E$ and the corresponding inner product  
are denoted by  $\|\cdot\|$ and $\langle\cdot,\cdot\rangle$, respectively.
Let 
\[
	\LL^n := \{\bp \in \E^n: \Delta^2\bp = 0\}
\]
denote the subspace of {\em linear chains} 
with $n$ points, and
$\LL^n_{\rm c}$ its 
orthogonal complement with respect to the inner product
\[
  \langle \bp,\bq \rangle_n := \sum_{i=1}^n \langle p_i,q_i \rangle
\]
on $\E^n$. The orthogonal projection 
$\Pi : \E^n \to \LL^n$ is mapping any subchain $\bp \in \E^n$ to its 
{\em linear component $\bp' := \Pi \bp \in \LL^n$}, i.e., $\bp - \bp' \in \LL^n_{\rm c}$.
Chains all of whose points are coincident are called {\em constant}. These chains, and 
also chains with constant linear component, are degenerate in some sense
and have to be kept away from certain arguments. To this end, we define the sets
\[
	\LL^n_* := \{\bp \in \LL^n : \Delta \bp \neq 0\}
	,\quad
	\E^n_* := \{\bp \in \E^n : \Delta \Pi \bp \neq 0\}
	,
\]
of non-constant linear chains and chains with non-constant linear component,
respectively.
Further, we define a norm and semi-norms on $\E^{\ge n}$ by
\[
  |\bP|_0 := \max_{0 \le i < \#\bP} \|p_i\|
  ,\quad
  |\bP|_j := |\Delta^j \bp|_0
  ,\quad
  j < n
  ,
\]
respectively.
The Euclidean norm of $\Pi$ equals $1$, and its $|\cdot|_0$-norm
can be only $\sqrt{n}$-times bigger, i.e.,
\begin{equation}
\label{eq:Pi}
	|\Pi|_0 := \max_{\bp \neq 0} \frac{|\Pi \bp|_0}{|\bp|_0}
	\le \sqrt{n}
	.
\end{equation}
To measure the deviation of a given chain from a linear one,
we introduce a notion, which will play a prominent role in the following.
\begin{definition}
The ratio
\[
  \kappa(\bp) := \begin{cases} 
    \frac{|\bp|_2}{|\Pi \bp|_1} & \text{if } \bp \in \E^n_* \\
    \infty                  & \text{else} 
  \end{cases}
\]
is called the {\em relative distortion} of $\bp\in \E^n$. 
More generally, we define
\[
  \kappa(\bP) := \max \{\kappa(\bp) : \bp \in T^n\bP\} 
	,\quad
	\bP \in \E^{\ge n}
  .
\]
\end{definition}
Chains $\bp \not\in\E^n_*$ with constant linear part are special
since they would cause a vanishing denominator, i.e.,
$|\Pi \bp|_1 = 0$. The formal setting $\kappa(\bp) = \infty$
will be used only to indicate that such chains do not satisfy
conditions of the form $\kappa(\bp) < \delta$.

A {\em similarity $S = (\varrho,Q,s) : \E \to \E$} is given by a scalar 
{\em scaling factor $\varrho > 0$}, 
an orthogonal {\em transformation matrix $Q\in \R^{d \times d}$},
and a {\em shift vector $s \in \E$}. 
It is acting on points in $\E$ according to $S(p) = \varrho p Q + s$.
The scaling factor $\varrho$ equals the norm of $S$, and we write
$|S| = \varrho$.
The group of similarities in $\E$ is denoted by $\calS(\E)$.
Application of $S$ to 
chains is understood as the application to
all points, i.e., $S(\bP) = [S(p_0); \dots;S(p_{N-1})]$. 
The sets
\[
	\E^n[\delta] := \{\bP\in \E^n : \kappa(\bP) \le \delta\}
	,\quad
	\E^{\ge n}[\delta] := \{\bP\in \E^{\ge n} : \kappa(\bP) \le \delta\}
\]
contain chains
with relative distortion bounded by some $\delta>0$.
The set $\E^n[\delta]$
is a cone in $\E^n$ and hence not compact. 
Similarities may be 
employed to reduce our investigations to some compact set $\Q^n[\delta]$,
to be derived as follows:

First, we note that the action of $\calS(\E)$
on $\LL^n_*$, the space of non-constant linear
chains, is transitive. Further, similarities commute with 
orthogonal projections. Together, these facts imply that
for any $\bp \in \E^n_*$
there exists $S_\bp \in \calS(\E)$ such that 
$\Pi S_\bp(\bp) = S_\bp(\Pi\bp) = \be$. 
The resulting chain $\bq = S_\bp(\bp)$ is called {\em normalized},
and the set of all normalized chains is denoted by
\[
  \Q^n := \{\bq \in \E^n_* : \Pi\bq = \be\}
	.
\]
Second, we have 
\(
\label{eq:rS=r}
  \kappa(S(\bp)) = 
	\frac{|S(\bp)|_2}{|S(\Pi \bp)|_1}
	=
	\frac{|S| |\bp|_2}{|S| |\Pi \bp|_1}
	= \kappa(\bp)
  ,\quad
  S \in \calS(\E)
  ,\
  \bp \in \E^n_*
  .
\)
Equally, $\kappa(\bp) = \infty$ implies $\kappa(S(\bp))=\infty$,
showing that the relative distortion is invariant with respect to
similarities. 
Combining these two observations, we see that
any chain $\bp$
with non-constant linear part is
similar to a normalized chain with equal relative
distortion,
\[
	\bp \in \E^n[\delta]
	\ \Leftrightarrow\
  S_\bp(\bp) \in 
	\Q^n[\delta] := \{\bq \in \Q^n: \kappa(\bq) \le \delta \}
	.
\]
Throughout, the letter $\bq$ is reserved for normalized chains,
and we use the abbreviation 
\[
  \bd := \bq-\be \in \LL^n_{\rm c}
\]
for the non-linear part of $\bq$
without further notice. The relative distortion of 
a normalized chain $\bq \in \Q^n$  is given by 
\(
\label{eq:kappa=d}
	\kappa(\bq) = \frac{|\be + \bd|_2}{|\Pi(\be+\bd)|_1} 
	= \frac{|\bd|_2}{|\be|_1} =
	|\bd|_2
	.
\)
So we recognize the set 
$\Q^n[\delta] = \be + \{\bd\in \LL^n_{\rm c}: |\bd|_2\le \delta\}$ 
as a lower-dimensional $|\cdot|_2$-ball 
centered at the standard linear chain $\be$.
Higher order differences are bounded by lower ones
according to the standard estimate
\(
\label{eq:standard}
  |\bp|_{i+j} \le 2^j |\bp|_i
	,\quad
	\bp \in \E^n
	,\
	i,j \in \N_0
	.
\)
Reverse estimates are possible 
on the subspace $\LL^n_{\rm c}$, where not only $|\cdot|_0$,
but also $|\cdot|_1$ and $|\cdot|_2$
are norms and hence equivalent. More precisely, 
for $j \in \{1,2\}$,
\[
  |\bd|_{j-1} \le \frac{n-j}{2}\, |\bd|_j 
	,\quad
	\bd \in \LL^n_{\rm c}
	.
\]
To show this, we define the matrices $M^1 \in \R^{n,n-1}$ and
$M^2 \in \R^{n-1,n-2}$ by
\[
  M^1_{i,j} = 
	\begin{cases}
	  \frac{j+1}{n}-1& \text{for } i \le j \\[.5ex]
		\frac{j+1}{n}& \text{for } i > j 
	\end{cases}
	,
	\quad
  M^2_{i,j} = 
	\begin{cases}
	  \frac{(j+1) (j + 2) (2 j + 3 - 3 n)}{n(1+n)(1-n)}-1 & \text{for } i \le j \\[.5ex]
		\frac{(j+1) (j + 2) (2 j + 3 - 3 n)}{n(1+n)(1-n)} &  \text{for } i > j 
	\end{cases}
	.
\]
Expediently using a computer algebra system, one verifies that
$\bd = M^1 \Delta \bd$ and $\Delta \bd = M^2 \Delta^2 \bd$
for $\bd \in \LL^n_{\rm c}$. The maximum norms of these matrices
are $|M^1|_\infty = (n-1)/2$ and $|M^2|_\infty = (n-2)/2$, as stated.

For later reference, we define $m := (n-1)/2$ and combine the
above estimates to the slightly weaker
inequality
\(
\label{eq:equinorm}
  |\bd|_i \le m^j |\bd|_{i+j} 
	,\quad
	i+j \le 2
	,\ 
	\bd \in \LL^n_{\rm c}
	.
\)
\section{GLUE-schemes}
In this section, we introduce the class of
subdivision algorithms to be analyzed in this work
and derive some of their basic properties. 
\begin{definition} 
\label{def:GLUE}
Given $m \in \N$, let $n := 2m+1$.
The function $\bG: \E^{\ge n} \to \E^{\ge n}$ defines a
{\em {\bf g}eometric, {\bf l}ocal, {\bf u}niform, {\bf e}quilinear subdivision scheme 
(or briefly GLUE-scheme) in $\E$ with spread $n$}
if $\#\bG(\bP) = 2\#\bP-n+1$ and
if it satisfies the following properties:
\begin{itemize}
\item[{\em (G)}]
$\bG$ commutes with similarities, i.e,
\(
\label{eq:invariant}
  \bG\circ S  = S  \circ \bG
  ,\quad
  S \in \calS(\E)
  .
\)
\item[{\em (L)}]
The points $p'_{2i}$ and $p'_{2i+1}$ of the chain 
$\bP' := \bG(\bP)$ depend only on
$p_i,\dots,p_{i+m}$.
\item[{\em (U)}]
There exist functions $g_0,g_1: \E^{m+1} \to \E$ independent of $i$ such that
\[
  p'_{2i+\lambda} = g_\lambda(p_i,\dots,p_{i+m})
  ,\quad	
  \lambda\in\{0,1\}
	,\
	0 \le i < \#\bP-m
  .
\]
These functions are $C^{1,\nu}$ in a neighborhood of $\bE_{m+1}$
for some $\nu > 0$, called the {\em regularity parameter} of $\bG$.
\item[{\em (E)}]
The standard linear chain $\be$ is scaled down and translated by $\bG$
according to
\(
\label{eq:E}
  \bG(\be) = (\bE_{n+1} +(m+\tau) e)/2
\)
for  some $\tau \in [0,1)$, called the {\em shift} of $\bG$. 
In particular, the scheme is called 
{\em primal} if $\tau=0$, and {\em dual} if $\tau = 1/2$.
\end{itemize}
\end{definition}
Let us briefly comment on this definition:
Commutation with similarities characterizes a subdivision process
which is independent of the scale or the orientation of the given data.
Linear subdivision schemes also commute with any element of the even larger group 
of affinities, provided that the weights sum to $1$. Thus, all these
schemes satisfy property (G).
Property (L) is crucial not only here but also in the standard theory
of linear schemes -- relatively little is known about schemes with 
global support like variational subdivision.
Also assuming that a single pair of rules according to (U) is applied everywhere
is customary as it captures most schemes of practical relevance. 
It is important to note that in (U) smoothness of the functions $g_0,g_1$
is assumed only in a neighborhood of the standard linear chain $\bE_{m+1}$.
Further away, these functions may be even discontinuous.
While reproduction of straight lines is a natural property of a geometric
subdivision algorithm, condition~(E) demands slightly more: equidistant points
on a line have to be mapped to equidistant points. 
As an example, consider the corner cutting scheme with weights $w$ and $1-w$,
as suggested by de~Rham \cite{deRham:1956}. Here,
only the special case $w = 1/4$, also known as 
Chaikin's algorithm \cite{Chaikin:1974},
satisfies property~(E) and thus a crucial prerequisite of our analysis. 

Our example CPS uses always $m+1=4$ old points to generate a new one.
Thus, its spread is $n=7$. The functions $g_0,g_1$ are smooth in a vicinity
of $\bE_4$, implying that the regularity parameter is $\nu=1$.

Repeated application of $\bG$ to the {\em initial chain $\bP$} yields the sequence
\[
  \bP^{\ell} := \bG^\ell(\bP)
  ,\quad
  \ell \in \N_0
  .
\]
Throughout, we assume that the number $N := \#\bP$ of initial points
is not smaller than the spread
of $\bG$, i.e., $N \ge n$. In this case, the length of $\bP^\ell$, given by
\[
  N^\ell := \#\bP^\ell = 2^\ell(N-n+1)+n-1
  ,\quad
  \ell \in \N_0
  ,
\]
is monotonically increasing.
Throughout, and even without explicit declaration,
the symbols $\bG$ and $\bP^\ell$ will represent a GLUE-scheme
and the subdivision iterates of some chain $\bP$, respectively,
according to the above definitions.
The points of $\bP^\ell$ and its
differences of order $k$ are denoted by $p^\ell_i$ and
$\Delta^k p^\ell_i$, respectively.

Property (G) implies invariance of the functions $g_0$ and $g_1$
according to $g_\lambda(S(\bp)) = S(g_\lambda (\bp)),$ $S \in \calS(\E)$.
Given any $s \in \E$, the similarity $S := (1/2,\id,s/2)$ 
satisfies $S(s) = s$. Hence, 
$g_\lambda(s,\dots,s) 
  = S(g_\lambda(s,\dots,s)) = g_\lambda(s,\dots,s)/2 + s/2
$,
showing that constant chains are reproduced,
\[   
  g_\lambda(s,\dots,s) = s
  ,\quad
  s \in \E
  .
\]
According to the representation of linear schemes
in terms of pairs of matrices, we define associated self-maps
$\bg_0,\bg_1:\E^n \to \E^n$ by
\[
  \bg_0(p_0,\dots,p_{n-1}) :=
  \begin{bmatrix}
    g_0(p_{0},\dots,p_{m})\ \\
    g_1(p_{0},\dots,p_{m})\\
    g_0(p_{1},\dots,p_{m+1})\\
    g_1(p_{1},\dots,p_{m+1})\\
  \vdots\\
    g_0(p_{m},\dots,p_{n-1})
  \end{bmatrix}
  ,\quad
  \bg_1(p_{0},\dots,p_{n-1}) :=
  \begin{bmatrix}
    g_1(p_{0},\dots,p_{m})\\
    g_0(p_{1},\dots,p_{m+1})\\
    g_1(p_{1},\dots,p_{m+1})\\
  \vdots\\
    g_0(p_{m},\dots,p_{n-1})\\
    g_1(p_{m},\dots,p_{n-1})
  \end{bmatrix}
  .
\]
Thus,
\[
  T^n_{2i+\lambda}\bG(\bP) = \bg_\lambda(T^n_i\bP)
  ,\quad
  \lambda \in \{0,1\}
  .
\]
More generally, let 
\[
  \calL^\ell := \{0,1\}^\ell
  ,\quad
  \calL := \bigcup_{\ell \in \N} \calL^\ell
  ,
\]
denote the set of {\em index vectors} of length $\ell \in \N$
and arbitrary length, respectively.
For
$\Lambda = [\lambda_1,\dots,\lambda_\ell] \in \calL^\ell$, we write 
$|\Lambda|:= \ell$ for its length, and
\[
  \bg_\Lambda := \bg_{\lambda_\ell} \circ \cdots \circ \bg_{\lambda_1}
\]
for the corresponding composition of the functions $\bg_0,\bg_1$.
Then
\[
  T^n_j \bP^\ell = \bg_\Lambda(T^n_i\bP)
  ,\quad
  \lambda \in \calL^\ell
  ,
\]
where $j = 2^\ell \bigl(i+\sum_{k=1}^{\ell} 2^{-k} \lambda_k\bigr)$.
This means that any subchain at level $\ell$ can be 
represented as the image of a subchain of the initial data,
\[
  T^n\bP^\ell = \{\bgL(\bp) : \bp \in T^n\bP,\ |\Lambda| = \ell\}
  .
\]
The functions $\bgl$ have the following basic properties:
First, (G) and (E) imply
\begin{gather}
\label{eq:inv_bgl}
  \bgL(S(\bp)) = S(\bgL(\bp)) \\
\label{eq:scale_bgl}
  \bgL(\be) = 2^{-\ell}(\be + \tau_\Lambda e)
\end{gather}
for some $\tau_\Lambda \in \R$.
Second, the behavior near $\be$ is characterized as follows:
\begin{lemma}
\label{lem:g(e)}
Denote the derivative of $\bgl$ at $\be$ by
$M_\lambda := D\bgl(\be)$.
Given $\ell \in \N$,
there exists $\delta_\ell>0$ such that,
for any $\Lambda = [\lambda_1,\dots,\lambda_\ell] \in \calL^\ell$, 
the function $\bgL$ is $C^{1,\nu}$ on $\Q^n[\delta_\ell]$.
The derivative of $\bgL$ at $\be$ is given by
$M_\Lambda := D\bgL(\be) = M_{\lambda_\ell} \cdots M_{\lambda_1}$.
In particular, there exists a constant
$c_\ell > 0$ such that
\begin{align}
\label{eq:D1}
  |\bgL(\bq) - \bgL(\be)|_i
  &\le
  c_\ell |\bd|_j
\\
\label{eq:D2}
  |\bgL(\bq) - \bgL(\be) - M_\Lambda\cdot \bd|_i
  &\le
  c_\ell\, |\bd|_j^{1+\nu}
\end{align}
for all $i,j \in \{0,1,2\}$, $\bq \in \Q^n[\delta_\ell]$, and $\Lambda \in \calL^\ell$.
\end{lemma}
\begin{proof}
By assumption, for all $\bp$ in a neighborhood of $\be$,
there exist linear maps $D\bgl(\bp) : \E^n \to \E^n$
such that
\[
  \lim_{\bh\to 0}
	\frac{|\bgl(\bp+\bh)-\bgl(\bp)-D\bgl(\bp)\cdot \bh|_0}{|\bh|_0} = 0
	.
\]
Hence, for arbitrary scaling factors $\varrho>0$ and shift vectors $t\in \E$,
invariance implies
\[
  \lim_{\bh'\to 0}
	\frac{|\bgl(\bp'+\bh')-\bgl(\bp')-D\bgl(\bp)\cdot \bh'|_0}{|\bh'|_0} = 0
\]
for $\bp' := \varrho \bp + t$ when setting $\bh' := \varrho \bh$.
That is, $\bgl$ is differentiable at $\bp'$ with $D\bgl(\bp') = D\bgl(\bp)$.
In particular, $\bgl$ is $C^{1,\nu}$ in a neighborhood of
$\varrho \be + t$ with $D\bgl(\varrho \be+t) = M_\lambda$.
In view of \eqref{eq:scale_bgl}, also the composed function $\bgL$
is $C^{1,\nu}$ in a neighborhood of $\be$. By the chain rule,
$M_\Lambda = M_{\lambda_\ell} \cdots M_{\lambda_1}$, validating
the two inequalities for $i=j=0$. 
For arbitrary $i,j$, the inequality follows from
applying the estimates \eqref{eq:standard} and \eqref{eq:equinorm}
to the left and right hand side, respectively.
\end{proof}
Here and below, we use the following conventions concerning
constants:
Indexed constants like $c_j,c'_j$
have fixed values,
while $c,c'$ are {\em generic constants} which
may change their value at every appearance.
Lower case constants
depend only on fixed parameters of
the subdivision scheme $\bG$ under consideration (like $n,\nu$, or $d$)
and on some parameter $z$ to be introduced in Section~\ref{sec:regular}.
Capitals like $C,C'$ also denote generic constants,
but play a different role. They appear in estimates on
sequences and may depend on anything but the sequence index.
For instance, in a typical expression like
$\kappa(\bP^\ell) \le C\, 2^{-\ell}, \ell \in \N$,
the constant $C$ may depend on the initial data $\bP^0$, but not on $\ell$.
%
%
\section{Straightening}
\label{sec:straight}
The key to assessing regularity of GLUE-schemes is  
an analysis of the behavior of relative distortion
as subdivision proceeds.
Because relative distortion may be infinite,
we have to deal with sequences with values in $\R \cup \{\infty\}$.
We say that a certain property holds for
{\em almost all} indices of such a sequence 
if there is only a finite number of exceptions. This convention
is useful if we are only interested in the long term behavior
of a sequence and want to avoid a special treatment
of a finite number of trailing infinite values.
In this spirit,
sequences are called {\em essentially bounded} or {\em essentially summable}
if they contain only a finite number of infinite values, and possess that property 
beyond some index $\ell_0$.
To illustrate the necessity of our conventions, 
consider CPS applied to the initial chain $\bP \in \E^7$ formed
by the points of a regular heptagon lying on the unit circle. 
For symmetry reasons, $\Pi \bP = [0;\dots;0]$
is a constant chain located at the origin. Hence, $\kappa(\bP) = \infty$,
while all subsequent iterates $\bP^\ell$ have finite relative distortion.
By our conventions, it makes sense to check the sequence 
$\kappa_\ell(\bP), \ell \in \N_0$
for essential boundedness or summability, disregarding the 
infinite value for $\ell = 0$.
\begin{definition}
\label{def:straight}
Let
\[
  \kappa_\ell(\bP) := \kappa(\bP^\ell)
  ,\quad 
  \ell \in \N
  .
\]
The chain $\bP \in \E^{\ge n}$ is said to be
\begin{itemize}
\item 
{\em straightened by $\bG$} if $(\kappa_\ell(\bP))_{\ell\in \N}$ is a null sequence;
\item 
{\em strongly straightened by $\bG$} if $(\kappa_\ell(\bP))_{\ell\in \N}$ is essentially summable;
\item 
{\em straightened by $\bG$ at rate $\alpha \in (0,1]$} if 
$(2^{\ell\alpha}\kappa_\ell(\bP))_{\ell\in \N}$ is essentially bounded.
\end{itemize}
\end{definition}
Clearly, straightening at some rate $\alpha >0$ implies strong straightening,
and strong straightening implies straightening. The different categories
of straightening are related to a decay of differences of points in
the following way:
\begin{lemma}
\label{lem:halvening}
If the chain $\bP \in \E^{\ge n}$ is 
\begin{itemize}
\item 
straightened by $\bG$, then $|\bP^\ell|_1 \le C q^\ell$ for any $q > 1/2$;
\item
strongly straightened by $\bG$, then $|\bP^\ell|_1 \le Cq^\ell$ for $q  = 1/2$;
\item 
straightened by $\bG$ at rate $\alpha$,
then $|\bP^\ell|_2 \le C\,  2^{-\ell(1+\alpha)}$
\end{itemize}
for some constant $C> 0$ and almost all $\ell \in \N$.
\end{lemma}
\begin{proof}
Based on $\delta_1$ as defined as in Lemma~\ref{lem:g(e)},
let $\delta_0 := \min(\delta_1, 1/(2m))$. 
By \eqref{eq:equinorm} and \eqref{eq:kappa=d},
$|\bd|_1 \le m|\bd|_2 = m \kappa(\bd) \le 1/2$. Further, by
\eqref{eq:D1},
\[
  \frac{|\bgl(\bq)|_1}{|\bq|_1} 
  \le
  \frac{|\bgl(\be)|_1 + |\bgl(\bq)-\bgl(\be)|_1}{|\be|_1 - |\bd|_1}
  \le
  \frac{1 + 2 c_1 |\bd|_2}{2(1 - m |\bd|_2)}
  =
  \frac{1 + 2 c_1 \kappa(\bq)}{2(1-m\kappa(\bq))}
\]
for $\bq \in \Q^n[\delta_0]$. 
Hence, there exists a constant $c$
such that
\[
  \frac{|\bgl(\bq)|_1}{|\bq|_1} 
  \le
	\frac{1}{2}\,
  (1 + c \kappa(\bq))
  ,\quad
  \bq \in \Q^n[\delta_0]
  .
\]
By invariance under similarities, we obtain
\[
	\frac{|\bgl(\bp)|_1}{|\bp|_1}
  \le 
	\frac{1}{2}\, (1+c\kappa(\bp)) 
  ,\quad
  \bp \in \E^n[\delta_0]
  .
\]
There exists $\ell_0$ such that $\kappa_\ell(\bP) \le \delta_0$
for all $\ell \ge \ell_0$.
For any subchain $\bp^\ell \in T^n \bP^\ell $
at level $\ell\ge \ell_0$ there exists a subchain $\bp^{\ell_0} \in T^n \bP^{\ell_0} $
at level $\ell_0$
and an index vector $\Lambda \in \calL^{\ell-\ell_0}$ such that
$\bp^\ell = \bgL(\bp^{\ell_0})$. This subchain can be estimated by repeated application of
the above inequality,
\[
  |\bp^\ell|_1 \le 
  |\bp^{\ell_0}|_1 \prod_{j=\ell_0}^{\ell-1}
  \frac{1}{2}\,(1+c \kappa_{j}(\bp)) 
  .
\]
Since $\bp^\ell$ was chosen arbitrarily, we obtain also
\[
  |\bP^\ell|_1 \le 
  |\bP^{\ell_0}|_1 \prod_{j=\ell_0}^{\ell-1}
  \frac{1}{2}\,(1+c \kappa_{j}(\bP)) 
  .
\]
First, if $\bP$ is straightened by $\bG$, 
we may choose $\ell_0$ even larger so that
$c\kappa_\ell(\bP) \le 2q-1$
for all $\ell \ge \ell_0$. Hence,
\[
  |\bP^\ell|_1  
	\le
	|\bP^{\ell_0}|_1 q^{\ell-\ell_0}
  ,
\]
verifying the first claim.

Second, if $\bP$ is strongly straightened by $\bG$,
\[
  2^{\ell} |\bP^\ell|_1 
  \le
  2^{\ell_0} |\bP^{\ell_0}|_1 \prod_{j=\ell_0}^{\infty}
  (1+c \kappa_{j}(\bP))
  = C
\]
for all $\ell \ge \ell_0$.
Convergence of the infinite product, and in particular
finiteness of the constant $C$, is guaranteed since the sequence $\kappa_{j}(\bP)$
is essentially summable. 

Third, if $\bP$ is straightened by $\bG$ at rate $\alpha$, we have
\[
  |\bp^\ell|_2 \le C 2^{-\ell\alpha} |\Pi \bp^\ell|_1
	\le C 2^{-\ell \alpha} |\bp^\ell|_1 
	\sup_{\bp \in \E^n[\delta_0]} \frac{|\Pi \bp|_1}{|\bp|_1}
\]
for some constant $C$. Since $\Pi$ commutes with similarities, we have
\(
\label{eq:Pi1}
  \sup_{\bp \in \E^n[\delta_0]} \frac{|\Pi \bp|_1}{|\bp|_1}
	=
	\sup_{\bq \in \Q^n[\delta_0]} \frac{|\Pi \bq|_1}{|\bq|_1}
	=
	\sup_{\bq \in \Q^n[\delta_0]} \frac{|\be|_1}{|\be+\bd|_1}
	\le
	\frac{1}{1-m\delta_0}
	\le 2
	.
\)
Thus, using the already proven result on strongly straightened chains,
we find
\[
  |\bP^\ell|_2 
	\le C 2^{-\ell \alpha + 1} |\bP^\ell|_1
	\le C' 2^{-\ell(\alpha+1)}
\]
for almost all $\ell \in \N$,
as requested.
\end{proof}
While the cases of straightening and strong straightening may be
of some theoretical interest, straightening at a certain rate $\alpha$
is most important for applications.
The rest of this sections deals with the question how to
establish specific values for $\alpha$ and $\delta$ such that
all chains
$\bP \in \E^{\ge n}[\delta]$ are straightened by $\bG$ at rate $\alpha$.
First, we show that it is sufficient to
consider finite levels of subdivision. To this end,
we define the functions $\Gamma_\ell, \ell \in \N$, by
\[
  \Gamma_\ell[\delta] := 
	\sup_{0 < |\bd|_2 \le \delta} \frac{\kappa_\ell(\be+\bd)}{|\bd|_2}
	,\quad
	\delta > 0
	.
\]
By \eqref{eq:kappa=d}, we obtain
$\kappa_\ell(\bq) \le \Gamma_\ell[\delta] \kappa(\bq)$
for $\bq \in \Q^n[\delta]\backslash\{\be\}$.
However, $\kappa(\be) = \kappa_\ell(\be)= 0$
so that this estimate is valid for all $\bq \in \Q^n[\delta]$.
Moreover, by invariance of relative distortion, we obtain
\[
  \kappa_\ell(\bp) \le \Gamma_\ell[\delta] \kappa(\bp)
	,\quad
	\bp \in \E^n[\delta]
	.
\]
If $\Gamma_\ell[\delta]\le \varepsilon/\delta$,
we have $\bgL(\bq) \in \E^n[\varepsilon]$. Hence, for any $\Sigma\in \calL^s$, 
\[
  \kappa(\bgS(\bgL(\bq))) \le \Gamma_s[\varepsilon] \Gamma_\ell[\delta] \kappa(\bq)
	,
\]
and we conclude that
\(
\label{eq:Gamma}
 \Gamma_{s+\ell}[\delta] \le \Gamma_{s}[\varepsilon]\, \Gamma_\ell[\delta]
	.
\)
Typically, it is hard to determine $\Gamma_\ell[\delta]$ explicitly.
Therefore, we
show how straightening rates can be derived from
upper bounds.
\begin{lemma}
\label{lem:straight}
If $\Gamma_\ell[\delta]\le \Gamma < 1$,
then all chains $\bP \in \E^{\ge n}[\delta]$ are 
straightened by $\bG$ at rate 
$\alpha = -\ell^{-1}\log_2 \Gamma$.
\end{lemma}
\begin{proof}
Analogous to the preceding proof, we find using 
\eqref{eq:Pi}, \eqref{eq:kappa=d}, \eqref{eq:equinorm},
\begin{align*}
  \kappa(\bgS(\bq)) &\le 
  \frac{|\bgS(\be)|_2 + |\bgS(\bq)-\bgS(\be)|_2}
			 {|\Pi\bgS(\be)|_1 - 2|\Pi(\bgS(\bq)-\bgS(\be))|_0}\\
  &\le
  \frac{2 c_s\, |\bd|_2}{1 - 4|\Pi|_0\, c_s\, |\bd|_0}
  \le
  \frac{2 c_s\, \kappa(\bq)}{1 - 4 \sqrt{n} m^2 c_s\, \kappa(\bq)}
\end{align*}
for any $\Sigma \in \calL^s$ and $\bq \in \Q^n[\delta_s]$.
With $\varepsilon_s := \min(\delta_s, 1/(8\sqrt{n}m^2 c_s))$, we obtain
\[
  \kappa(\bgS(\bq)) \le 4 c_s \kappa(\bq)
  ,\quad
  \bq \in \Q^n[\varepsilon_s]
  ,
\]
and hence $\Gamma_s[\varepsilon_s] \le 4 c_s$.
Let $\varepsilon'_\ell := \min_{s<\ell} \varepsilon_s$ and
$c'_\ell := 4 \max_{s<\ell} c_s$.
By \eqref{eq:Gamma}, $\Gamma_{r\ell}[\delta] \le (\Gamma_\ell[\delta])^r$.
That is, we can choose $r_0 \in \N$ such that 
$\Gamma_{r\ell}[\delta] \le \varepsilon'_\ell/\delta$ for all $r \ge r_0$.
Given $j \ge \ell r_0$, there exists $s < \ell$ and $r \ge r_0$ such that
$j = r \ell + s$. Using \eqref{eq:Gamma} again, 
\[
  \Gamma_{j}[\delta] \le \Gamma_{s}[\varepsilon'_\ell]\, \Gamma_{r \ell}[\delta]
	\le c'_\ell\, (\Gamma_\ell[\delta])^r
	\le c'_\ell\, \Gamma^r
	.
\]
With $\Gamma = 2^{-\alpha \ell}$, we finally obtain
\(
\label{eq:kappa_j}
  2^{\alpha j} \kappa_j(\bP) \le
	2^{\alpha j} \Gamma_{j}[\delta] \kappa(\bP)
	\le
	c'_\ell\, 2^{\alpha j} \Gamma^r\kappa(\bP)
	\le
	c'_\ell\, 2^{\alpha(j-r\ell)} \kappa(\bP)
	\le c'_\ell\, 2^{\alpha \ell} \delta
\)
for $\bP \in \E^{\ge n}[\delta]$ and
almost all $j \in \N$.
The right hand side is bounded independent of $j$, as requested, which finished the proof.
\end{proof}
As the function $\Gamma_\ell$ is monotonically increasing, smaller values
of $\delta$ suggest higher H{\"o}lder exponents at the account of
a smaller range of applicability. The next theorem resolves this conflict.
It shows how to combine good local bounds with coarser estimates on a larger set
in a beneficial way:
\begin{theorem}
\label{thm:straight+}
Let 
\[
	\Gamma_k[\delta,\gamma] 
	:= 
	\max_{\delta \le |\bd|_2 \le \gamma} \frac{\kappa_k(\be+\bd)}{|\bd|_2}
	,\quad
	0 < \delta < \gamma
  ,\quad
  k \in \N
	.
\]
If $\Gamma_k[\delta,\gamma] < 1$ and
$\Gamma_\ell[\delta] \le \Gamma < 1$, then
all chains $\bP \in \E^{\ge n}[\gamma]$ are 
straightened by $\bG$ at rate 
$\alpha = -\ell^{-1}\log_2 \Gamma$.
\end{theorem}
\begin{proof}
Let $\bp \in T^n\bP$, $s\in \N$ be sufficiently large, and
$\Sigma \in \calL^s$ be some index vector.
Choose a partition $\Sigma = [\Sigma_2,\Sigma_1]$ such that
$\Sigma_1 \in \calL^{s_1}$ is the shortest index vector with 
$\kappa(\bg_{\Sigma_1}(\bp)) \le \delta$. Because relative distortion is reduced
at least by the factor $\eta:=\Gamma_k[\delta,\gamma]$ by always $k$ steps of subdivision, 
the length $s_1$ of $\Sigma_1$
cannot exceed the value $s^*_1~:=~k \log_{\eta}(\delta/\gamma)$.
Let $r_0$ be defined as in the proof of Lemma~\ref{lem:straight},
and assume that $s \ge s^*_1 + r_0\ell$. Then the length
$s_2 = s - s_1$ of $\Sigma_2$ is at least $r_0\ell$ and we can
use \eqref{eq:kappa_j} to estimate the relative distortion of
$\bp' := \bg_{\Sigma_1}(\bp) \in \E^n[\delta]$. We obtain
\[
  2^{\alpha s} \kappa(\bgS(\bp))
	= 2^{\alpha s_1} 2^{\alpha s_2} \kappa(\bg_{\Sigma_2}(\bp'))
	\le 
	2^{\alpha s^*_1} 2^{\alpha \ell} \delta 
	.
\]
Since $\bp$ and $\Sigma$ were chosen arbitrarily, it follows that 
$2^{\alpha s} \kappa_s(\bP)$ is bounded by some constant,
\[
  2^{\alpha s} \kappa_s(\bP) \le 2^{\alpha (s^*_1+\ell)} \delta
	,\quad
	s \ge s^*_1 + r_0\ell
	,
\]
and the proof is complete.
\end{proof}
Typically, the functions $g_0,g_1$ and all other functions appearing here
can be coded for numerical evaluation in terms of standard library functions.
Using interval arithmetics, 
the range of such functions over compact intervals can be
estimated reliably and efficiently by a computer.
In this respect, the above bound $\Gamma_k[\delta,\gamma]<1$ admits automated verification
for given values of $\delta$ and $\gamma$.
However, 
when trying to determine some bound $\Gamma > \Gamma_\ell[\delta]$,
the situation is more complicated.
The point is that the domain $\{\bd : 0 < |\bd|_2 \le \delta\}$,
which is used to define $\Gamma_\ell[\delta]$, is not compact.
From an application point of view, it may be sufficient to
evaluate the ratio $\kappa_\ell(\bq)/|\bd|_2$ at a sufficiently dense set of chains.
However, considering the vanishing denominator for $\bd=0$,
the determination of safe bounds requires more care.
The following lemma provides an upper bound on $\Gamma_\ell[\delta]$
in terms of the range of continuous functions over a compact domain.
Thus it becomes possible to establish a rigorous upper bound $\Gamma$
on $\Gamma_\ell[\delta]$ for given $\delta$ with the help of a computer.
The result will also prove to be useful 
for the further development of theory in Section~\ref{sec:asymptotic}.
Below,
\[
  |M|_2 := \max_{|\bd|_2=1} |M\cdot\bd|_2
\]
denotes the norm of the linear operator $M : \LL^n_{\rm c} \to \E^n$
with respect to the $|\cdot|_2$-norm.
\begin{lemma}
\label{lem:GammaBnd}
It is
\[ 
	\Gamma_\ell[\delta] \le
	\Gamma_\ell^*[\delta] := 
	\max_{\Lambda \in \calL^\ell}
  \frac{\max_{\bq \in \Q^n[\delta]} | D\bgL(\bq)|_2} 
       {\min_{\bq \in \Q^n[\delta]} |\Pi\bgL(\bq)|_1}
	.
\]
\end{lemma}

\begin{proof}
Given $\bq= \be+\bd\in \Q^n[\delta]$, 
let $\varphi(\tau) := \bgL(\be+\tau\bd)-\bgL(\be), \tau \in [0,1]$.
Then $\varphi(1)-\varphi(0) = \int_0^1\varphi'(\tau)\, d\tau$ yields
\[
  |\bgL(\bq)|_2 = |\varphi(1)-\varphi(0)|_2
	\le
	\int_0^1 |D\bgL(\be+\tau\bd)|_2\, |\bd|_2\, d\tau
	\le
	\max_{\bq \in \Q^n[\delta]} | D\bgL(\bq)|_2 |\bd|_2
	.
\]
Now,
\[
  \frac{\kappa(\bgL(\bq))}{\kappa(\bq)} =
	\frac{|\bgL(\bq)|_2}{|\Pi\bgL(\bq)|_1 |\bd|_2}
	\le
	\frac{\max_{\bq \in \Q^n[\delta]} | D\bgL(\bq)|_2} 
       {\min_{\bq \in \Q^n[\delta]} |\Pi\bgL(\bq)|_1}
	,
\]
and the claim follows.
\end{proof}
Together, Theorem~\ref{thm:straight+} and Lemma~\ref{lem:GammaBnd}  
facilitate an automated assessment of straightening properties of 
a given GLUE-scheme. 
First, an upper bound $\Gamma$ 
on $\Gamma_\ell^*[\delta]$
is determined for a relatively small value of $\delta$. Then, a
preferably large value $\gamma$ is sought by checking the condition
$\Gamma_k[\delta,\gamma]<1$. 
In general, larger values for $\ell$ and $k$ yield better results
at the cost of more time-consuming computations.

The requested software consists of a universal control unit
operating on a specific implementation of the functions $g_0,g_1$. 
Even the requested derivatives need not be coded explicitly
when using a tool for automated differentiation, as it comes with 
many packages for interval arithmetics. We will describe such
a procedure in a forthcoming report.

\section{Smoothness of limit curves}
\label{sec:regular}
In this section, we relate straightening of chains
to smoothness properties of corresponding limit curves. 
We show that straightening and strong straightening imply 
continuity and differentiability of the natural parametrization of
the limit curve, respectively.
Straightening at rate $\alpha$ yields local H{\"o}lder continuity
of the first derivative with according exponent. 
Further, strong straightening yields a regular limit curve, 
i.e., the first derivative vanishes nowhere.
The limit curve corresponding to some initial chain is expressed
as the limit of a sequence of smooth parametrized curves, which are
defined as the linear combination of the points at increasing
levels of subdivision with uniform dyadic shifts of a given base function. 
The natural domain of the limit curve corresponding to an
initial chain with $\#\bP=N$ points is the interval $[0,N-n+1]$,
where $n$ is the spread of the GLUE-scheme in use.
However, for technical reasons, we consider smoothness properties
only on the open interval $I := (0,N-n+1)$. 
Uniform convergence of certain function sequences can 
only be observed on compact subintervals $I_z := [z,N-n+1-z]$, where 
$z>0$ is always understood as an arbitrary,
but fixed number. In this way, convergence results obtained on 
$I_z$ typically transfer to all of $I$.

The max-norm of a continuous curve $\Phi : I \to \E$ on the interval $I_z$
is defined by
\[
  |\Phi|_z := \max_{t \in I_z} \|\Phi(t)\|
	,
\]
where $\|\cdot\|$ is the Euclidean norm on $\E$, as above.
Throughout, limits
of sequences of curves are understood with respect to this norm.
Also the max-norm of real-valued functions on $I_z$ is denoted by $|\cdot|_z$.

Differentiation of a curve $\Phi: I \to \E$ with respect to its parameter
is expressed by means of the operator $\partial : \Phi \mapsto \Phi'$.
A curve is called
$C^k$ if the $k$th 
derivative $\partial^k\Phi$ exists and is a continuous function on $I$.
It is called
$C^{k,\alpha}$ if, moreover, 
\[
  \sup_{0 < h \le h_0} h^{-\alpha} \omega_z(\partial^k\Phi,h)
	< \infty
\]
for some constant $h_0>0$ and any $z>0$, where
\[
  \omega_z(\partial^k\Phi,h) := |\partial^k\Phi(\cdot+h/2)-\partial^k\Phi(\cdot-h/2)|_{z+h/2}
\]
is the modulus of continuity of $\partial^k\Phi$. 
Taking the norm on the interval $I_{z+h/2}$ guarantees that only values
on $I_z$ are taken into account.
The range of $h$ is bounded from above by $h_0$
because we intend to use the modulus of continuity as a local measure of smoothness,
disregarding global growth.
\begin{definition}
Let $k \in \N_0$ and $\alpha \in (0,1]$.
The GLUE-scheme $\bG$ is called 
\begin{itemize} 
\item 
{\em convergent at $\bP$}
if there exists a continuous {\em limit curve}
$\Phi[\bP] : \R \to \E$ such that
\(
\label{eq:limfun}
  \lim_{\ell \to \infty} \sup_{i\in I^\ell_z} 
	\bigl\|\Phi[\bP](2^{-\ell}i)-p^\ell_i\bigr\| = 0
\)
for any $z>0$,
where $p^\ell_i$ are the points of the $\ell$-th iterate $\bP^\ell$,
and $I^\ell_z := \N \cap 2^\ell I_z$ is the set of indices~$i$ satisfying 
$2^{-\ell}i \in I_z$;
\item
{\em $C^k$ or $C^{k,\alpha}$ at $\bP$} if
it is convergent at $\bP$, and $\Phi[\bP]$ is $C^k$ or $C^{k,\alpha}$, respectively;
\item
{\em almost $C^{k,\alpha}$ at $\bP$} if
it is $C^{k,\alpha'}$ at $\bP$ for any $\alpha'<\alpha$.
\end{itemize}
\end{definition}
The main theorem to be proven in this section is the following:
\begin{theorem}
\label{thm:C1alpha}
Let $\bG$ be a GLUE-scheme with spread $n$, and let $\bP \in \E^{\ge n}$. If
\begin{itemize}
  \item 
$\bP$ is straightened by $\bG$, then $\bG$ is convergent at $\bP$;
  \item 
$\bP$ is strongly straightened by $\bG$, then $\bG$ is $C^1$ at $\bP$;
  \item 
$\bP$ is straightened by $\bG$ at rate $\alpha$, then $\bG$ is $C^{1,\alpha}$ at $\bP$.
\end{itemize}
\end{theorem}
Before we can turn to the proof, 
we have to further prepare the ground.
The real-valued function $\varphi : \R \to \R$ is called a {\em generator},
if it is continuous, has compact support, and forms a partition of unity 
according to
\[
	\sum_{j \in \Z}\ \varphi(\cdot-j) = 1
  .
\]
Given such a generator and a chain $\bP \in \E^{\ge n}$,
we define the corresponding curve
\[
  \Phi^\ell[\bP,\varphi] := \sum_{j=0}^{\#\bP-1} p_j \varphi(2^\ell \cdot - j)
\]
at level $\ell \in \N_0$.
The points $p_j \in \E$ are also called the {\em control points}
of the curve. Typically, we will consider curves at level $\ell$ corresponding
to chains at the same level, i.e., expressions of the form 
$\Phi^\ell[\bP^\ell,\varphi]$. 
Choose $r$ such that $\operatorname{supp} \varphi \subseteq [-r,r]$
and let $\ell_z(\varphi) \in \N$ be an upper bound on $\log_2(r/z)$.
Then, for levels $\ell \ge \ell_z(\varphi)$, 
the sum $\Phi^\ell[\bP^\ell,\varphi] = \sum_{j=0}^{N^\ell-1} p_j^\ell \varphi(2^\ell \cdot-j)$
is complete in the following sense: 
for all indices $j \in \Z$ with $j < 0$ or $j \ge N^\ell$,
the functions $\varphi(2^\ell \cdot - j)$ vanish identically on $I_z$.
Equally, it holds
\(
\label{eq:sum1}
  \sum_{j=0}^{N^\ell-1} \varphi(2^\ell x-j) = 1
  ,\quad
  x \in I_z
  ,\
  \ell \ge \ell_z(\varphi)
  .
\)
Further,
with $\bar\varphi := \max_t \sum_j |\varphi(t - j)|$ the
{\em Lebesgue constant} of $\varphi$, we obtain the estimate
\(
\label{eq:Lebesgue}
  \bigl|\Phi^\ell[\bP,\varphi]\bigr|_z \le \bar\varphi |\bP|_0
	.
\)
The following lemma shows that it is possible to define $\Phi[\bP]$
as the limit of a sequence $\Phi^\ell[\bP^\ell,\varphi]$ of curves,
which are as smooth as the chosen $\varphi$.
Compared with the usual approximation by piecewise linear functions,
this approach simplifies the forthcoming arguments significantly.
\begin{lemma}
\label{lem:C0}
Let $\varphi$ be a generator.
If the function sequence $(\Phi^\ell[\bP^\ell,\varphi])_{\ell \in \N}$
is convergent and if $|\bP^\ell|_1$ is a null sequence, then 
$\Phi[\bP]:= \lim_{\ell \to \infty}\Phi^\ell[\bP^\ell,\varphi]$ is the limit
curve corresponding to $\bP$. 
In particular, $\bG$ is convergent at $\bP$. 
\end{lemma}
\begin{proof}
Denote the two summands on the right hand side of the estimate
\[
  \bigl\|\Phi[\bP](2^{-\ell}i)-p^\ell_i\bigr\|
  \le
  \bigl| \Phi[\bP] - \Phi^\ell[\bP^\ell,\varphi] \bigr|_z +
  \bigl\| \Phi^\ell[\bP^\ell,\varphi](2^{-\ell}i)-p^\ell_i \bigr\|
  .
\]
by $s^\ell$ and $s^\ell_i$ and assume $\ell \ge \ell_z(\varphi)$. 
Using \eqref{eq:sum1} and \eqref{eq:Lebesgue},
we see that the second one is bounded by
\[
  s^\ell_i
  =
  \Bigl\| \sum_{j=0}^{N^\ell-1} (p^\ell_j-p^\ell_i) \varphi(i-j) \Bigr\|
  \le
  \bar\varphi \sup_{|i-j|\le r} \|p^\ell_j-p^\ell_i\|
  \le 
  \bar\varphi r\, |\bP^\ell|_1
  .
\]
Since, by assumption, $s^\ell$ and $|\bP^\ell|_1$ are null sequences,
we obtain
\[
  0 \le 
  \lim_{\ell \to \infty} \sup_{i \in I^\ell_z} \bigl\| \Phi[\bP](2^{-\ell}i)-p^\ell_i \bigr\|
  \le
  \lim_{\ell \to \infty} 
  (s^\ell + \bar\varphi r \, |\bP^\ell|_1) = 0
  ,
\]
showing that $\Phi[\bP]$ satisfies \eqref{eq:limfun}.
\end{proof}
A GLUE-scheme $\bA$ is called {\em linear} if the functions $g_0,g_1$
in Definition~\ref{def:GLUE}, now renamed as $a_0,a_1$, have the form
\[
  a_\lambda(p_i,\dots,p_{i+m}) = \sum_{j=0}^m a_{\lambda,j} p_{i+j}
  ,\quad
  \lambda\in \{0,1\}
  ,
\]
for certain real weights $a_{\lambda,j}$ summing up to one. 
The associated self-maps $\bg_0,\bg_1$, now renamed as $\ba_0,\ba_1$, 
are given by a pair of $(n \times n)$-matrices $A_0,A_1$,
\[
  \ba_\lambda(T_i^n\bP) = A_\lambda T_i^n \bP
  ,\quad
  \lambda\in \{0,1\}
  .
\]
The analysis of linear
schemes is well-known \cite{StandardDyn02,Sabin:2010}, and we recall only a few facts which 
are needed in the following:
Assuming $\bA$ as convergent, we define
the {\em basic function $\psi := \Phi[\chi]$ of $\bA$} as subdivision limit
of the real-valued delta-sequence $\chi := (\delta_{0,i})_{i \in \Z}$.
This function is known to be a generator.
The {\em refinement equation} reads
\(
\label{eq:refine}
  \Phi^{\ell+r}[\bA^r\bP,\psi] = \Phi^\ell[\bP,\psi]
  ,\quad
  \ell,r \in \N
  .
\)
Products of the matrices $A_0,A_1$ are denoted by 
$A_\Lambda := A_{\lambda_\ell} \cdots A_{\lambda_1}$.
With $|A_\Lambda|_0$ the max-norm of the matrix $A_\Lambda$, we define
\(
\label{eq:rho}
  \varrho_\ell(\bA) := \max_{\Lambda \in \calL^\ell} |A_\Lambda|_0^{1/\ell}
	,\quad
	\varrho(\bA) := \limsup_{\ell \to \infty} \varrho_\ell(\bA)
	.
\)
That is, $\varrho(\bA)$ is the {\em joint spectral radius} of the matrices $A_0,A_1$.

We say that $\bA$ is {\em almost $C^{k,\alpha}$} if it is
almost $C^{k,\alpha}$ at all $\bP\in \E^{\ge n}$ and non-degenerate in the sense
that a constant limit can only be attained for constant initial data. In this case,
\begin{itemize} 
\item 
for $j = 1,\dots,k+1$, 
there exists a {\em difference scheme}
$\bA_j$ of order $j$
satisfying
\[
  \Delta^j \bA^\ell \bP = \bA_j^\ell \Delta^j \bP
  ,\quad
  \ell \in \N
  ;
\]
\item 
there exists a constant $c$ such that
\(
\label{eq:A_j}
  |\bA_j^\ell \bP|_0 \le c 2^{-\ell j}\, |\bP|_0
  ,\quad	
  \ell \in \N
  ,\
  j \le k
  ;
\)
\item
the difference scheme $\bA_{k+1}$ satisfies 
\(
\label{eq:jsr}
	\varrho(\bA_{k+1}) \le 2^{-k-\alpha}
	;
\)
\item
for $\ell = 1,\dots,k$,
the basic function $\psi_\ell$ of the 
{\em divided difference scheme} $\bbA_\ell := 2^\ell\bA_\ell$ 
of order $\ell$ are generators. In particular,
\(
\label{eq:partition}
  \sum_{j \in \Z} \psi_\ell(\cdot - j) = 1
	.
\)
\end{itemize}
If $\bA$ is $C^k$, then the $j$th derivative
of $\Phi^\ell[\bP^\ell,\psi]$ is given by 
\(
\label{eq:diff}
  \partial^j \Phi^\ell[\bP,\psi] =
  \Phi^\ell[2^{\ell j}\, \Delta^j \bP, \psi_j]
  ,\quad
  \ell \in \N
  ,\
  j  \le k
  ,
\)
where $\psi_j$ is the basic function of $\bar\bA_j$.

As an example, and for later use, we consider a family of linear
GLUE-schemes $\bA^\tau, \tau \in [0,1)$, with spread $n=5$, given by
\begin{align*}
  a_0^\tau(p_i,p_{i+1},p_{i+2}) &= \bigl((4-3\tau)p_{i} + (4+2\tau)p_{i+1} + \tau p_{i+2}\bigr)/8 \\
  a_1^\tau(p_i,p_{i+1},p_{i+2}) &= \bigl((1-\tau)     p_{i} + (6-2\tau)p_{i+1} + (1+3\tau)p_{i+2}\bigr)/8
  .
\end{align*}
It is easily verified by inspection that the shift of $\bA^\tau$ is $\tau$.
For $\tau=0$ and $\tau=1/2$, we recover
cubic and quartic B-spline subdivision, which
are known to be $C^{2,1}$ and $C^{3,1}$, respectively,
Otherwise, for $\tau \in (0,1)$, we consider 
the derived scheme $\bA^\tau_4$ for fourth differences, given 
by the $(1 \times 1)$-matrices
\[
  A^\tau_{4,0} = \frac{\tau}{8}
  ,\quad
  A^\tau_{4,1} = \frac{1-\tau}{8}
  .
\]
The joint spectral radius $\varrho(A^\tau_{4,0},A^\tau_{4,1}) = \max(\tau,1-\tau)/8$
is less than $1/8$, showing that the scheme
is almost $C^{3,\alpha}$ with $\alpha:= -\log_2 \max(\tau,1-\tau)$.

General GLUE-schemes can be analyzed with the aid of
linear schemes satisfying a proximity condition of the form
\[
  \bG(\bP) = \bA \bP + \bR(\bP)
\]
with some suitably bounded remainder $\bR$. The following lemma
is crucial in that respect.
\begin{lemma}
\label{lem:proxy}
Let $\bG$ be a GLUE-scheme, $\bA$ a convergent linear scheme with basic limit function $\psi$,
and $\bR := \bG-\bA$
the corresponding remainder. If $\bA$ is $C^k$, then 
\[
  \bigl|\partial^j(\Phi^{\ell+r}[\bP^{\ell+r},\psi] -
	\Phi^{\ell}[\bP^{\ell},\psi])\bigr|_z \le
  c\, \sum_{i=\ell}^\infty 2^{ij}|\bR(\bP^i)|_0
  ,\quad
  \ell,r \in \N
  ,
  j \le k
	,
\]
for some constant $c$.
\end{lemma}
\begin{proof}
The formula
\(
\label{eq:G-A}
  \bG^{r}(\bP^\ell) - \bA^{r}\bP^\ell = 
  \sum_{i=\ell}^{\ell+r-1} \bA^{r+\ell-i-1} \bR(\bP^{i})
\)
is easily verified by induction on $r$.
By \eqref{eq:refine}, \eqref{eq:diff}, \eqref{eq:A_j}, and \eqref{eq:standard},
\begin{align*}
  \bigl|\partial^j&(\Phi^{\ell+r}[\bP^{\ell+r},\psi] -
	\Phi^{\ell}[\bP^{\ell},\psi])\bigr|_z 
	=
  2^{(\ell+r)j}\bigl|\Phi^{\ell+r}[\Delta^j(\bG^r(\bP^\ell)-A^r\bP^{\ell}),\psi_j]\bigr|_z\\
	&\le 
	\bar\psi_j  2^{(\ell+r)j}  \bigl|\Delta^j(\bG^r(\bP^\ell)-A^r\bP^{\ell})\bigr|_0
	\le
	\bar\psi_j  2^{(\ell+r)j} \sum_{i=\ell}^{\ell+r-1}\bigl|\bA^{r+\ell-i-1}_j \Delta^j\bR(\bP^{i})\bigr|_0 \\
	&\le
	c'\, \sum_{i=\ell}^{\ell+r-1} 2^{(i+1)j} |\bR(\bP^{i})|_j
	\le
	c \, \sum_{i=\ell}^{\infty} 2^{ij} |\bR(\bP^{i})|_0
	.
\end{align*}
\end{proof}
Now, we are ready to prove Theorem~\ref{thm:C1alpha}:
\medskip

\begin{proof}
Let $\bA$ be a linear $C^1$-scheme with the same shift $\tau$ as $\bG$. 
For instance, we may choose
$\bA = \bA^\tau$, as introduced above. By \eqref{eq:D1}, \eqref{eq:equinorm} and \eqref{eq:Pi1},
\begin{align*}
  |\bgl(\bq) &- A_\lambda \bq |_0 
	\le
  |\bgl(\bq) - \bgl(\be)|_0 + |A_\lambda \bq - A_\lambda \be|_0\\
	&\le
	(c_1 + m^2 |A_\lambda|_0) |\bd|_2
	=
	c \kappa(\bq) |\Pi \bq|_1
	\le
	2 c\kappa(\bq) |\bq|_1
\end{align*}
for $\bq \in \Q^n[\delta_0]$ with $\delta_0$ as in the proof of Lemma~\ref{lem:halvening}.
Invariance under similarities yields
\[
	|\bgl(\bp) - A_\lambda \bp |_0 \le c' \kappa(\bp)|\bp|_1
	,\quad
	\bp \in \E^n[\delta_0]
	.
\]
If $\bP$ is straightened, 
then $\bP^i \in \E^{\ge n}[\delta_0]$ for almost all $i$.
Hence, using Lemma~\ref{lem:halvening},
the remainder $\bR := \bG-\bA$ 
is bounded by 
\[
	|\bR(\bP^i)|_0 \le c'\kappa(\bP^i) |\bP^i|_1
	\le C \kappa(\bP^i) q^i
\]
for almost all $i$.
Here, we may choose $q  = 2/3$ 
if $\bP$ is straightened, and
$q= 1/2$ if $\bP$ is strongly straightened by $\bG$.
Denoting the basic function of $\bA$ by $\psi$,
Lemma~\ref{lem:proxy} yields
\(
\label{eq:DPhi}
  \bigl|\partial^j(\Phi^{\ell+r}[\bP^{\ell+r},\psi] -
        \Phi^{\ell}[\bP^{\ell},\psi])\bigr|_z 
	\le  C 
  \sum_{i=\ell}^\infty 2^{ij} q^i \kappa(\bP^i)
	,\quad
	j \in \{0,1\}
  ,
\)
for almost all $\ell,r \in \N$.

First,
if $\bP$ is straightened by $\bG$, we consider the
case $j=0$. The sum tends to $0$ as $\ell \to \infty$ 
because $q =2/3$ and $\kappa(\bP^i)$ is essentially bounded.
Hence, $\Phi^{\ell}[\bP^{\ell},\psi]$
is a Cauchy sequence on $I_z$. 
Moreover, by Lemma~\ref{lem:halvening}, $|\bP^\ell|_1$
is a null sequence. Hence, 
by Lemma~\ref{lem:C0},
$\Phi[\bP] = \lim_{\ell \to \infty}\Phi^{\ell}[\bP^{\ell},\psi]$ 
is the limit of subdivision. In particular, $\bG$ is convergent at $\bP$.

Second, 
if $\bP$ is strongly straightened by $\bG$, we consider the
case $j=1$. Now, $q=1/2$, and the sum tends to $0$ as $\ell \to \infty$ 
because $\kappa(\bP^i)$ is essentially summable.
Hence, $\partial\Phi^{\ell}[\bP^{\ell},\psi]$
is a Cauchy sequence, too, showing that the limit curve $\Phi[\bP]$ is $C^1$.

Third, if $\bP$ is straightened by $\bG$ at rate $\alpha$, there
is a constant $C$ such that $\kappa(\bP^i) \le C 2^{-\alpha i}$
for almost all $i$.
Given $h\in(0,h_0]$, choose $\ell\in \N$ such that
$2^{-\ell} < h \le 2^{-\ell+1}$. If $h_0$ is sufficiently small, 
$\ell$ is sufficiently large to guarantee validity of the
estimates above.
The modulus of continuity of the derivative of the limit curve is bounded by
\[
  \omega_z(\partial\Phi[\bP],h) \le 
  \omega_z(\partial(\Phi[\bP]-\Phi^\ell[\bP^\ell,\psi]),h)
  +
  \omega_z(\partial\Phi^\ell[\bP^\ell,\psi],h)
  .
\]
By \eqref{eq:DPhi},
the first summand is bounded by
\begin{align*}
  \omega_z(\partial(\Phi[\bP]-&\Phi^\ell[\bP^\ell,\psi]),h)
  \le
  2\, \bigl| \partial(\Phi[\bP]-\Phi^\ell[\bP^\ell,\psi]) \bigr|_z 
	\le
	2c\, \sum_{i=\ell}^\infty \kappa(\bP^i) \\
	& \le
	2cC
  \sum_{i=\ell}^\infty 2^{-i\alpha} 
	=
  C' 2^{-\ell\alpha} 
	\le 
	C' h^\alpha
  .
\end{align*}
By the mean value theorem and equations \eqref{eq:diff}, \eqref{eq:Lebesgue},  
the second summand is bounded by
\[
  \omega_z(\partial\Phi^\ell[\bP^\ell,\psi],h)
  \le
  h \bigl|\partial^2 \Phi^\ell[\bP^\ell,\psi]\bigr|_z
  \le
  h \bigl|\Phi^\ell[ 2^{2\ell}\Delta^2 \bP^\ell,\psi_2]\bigr|_z 
  \le
  |\bP^\ell|_2 \, h 2^{2\ell} \bar \psi_2
	.
\]
Estimating $|\bP^\ell|_2$ by means of Lemma~\ref{lem:halvening}
and using $2^\ell \le 2/h$,
we obtain
\[
  \omega_z(\partial\Phi^\ell[\bP^\ell,\psi],h)
  \le
	c\, h2^{\ell(1-\alpha)}
	\le
	2 c h^\alpha
	.
\]
Together, $h^{-\alpha}\omega_z(\partial\Phi[\bP],h) \le C $ for some constant $C$, 
showing that $\bG$ is $C^{1,\alpha}$ at $\bP$.
\end{proof}
Strong straightening does not only imply differentiability of the limit
curve, but also its regularity in the sense of differential geometry.
\begin{theorem}
\label{thm:regular}
If the chain $\bP \in \E^{\ge n}$ 
is strongly straightened by the GLUE-scheme $\bG$, then
\[
  \partial \Phi[\bP](t) \neq 0
	,\quad
	t \in I
	.
\]
\end{theorem}
\begin{proof}
With $\psi^\tau$ the basic limit function
of the linear scheme $\bA^\tau$, we have
$\partial\Phi^\ell[\bP^\ell,\psi^\tau] = \Phi^\ell[\Delta\bbP^\ell,\psi^\tau_1]$,
where $\bbP^\ell := 2^\ell \bP^\ell$.
The coefficients of the derived scheme $\bA^\tau_1$ are non-negative
so that the corresponding basic function $\psi^\tau_1$ is non-negative, too.
Its support has length $4$ so that each point on the curve
$\Phi^\ell[\Delta\bbP^\ell,\psi^\tau_1]$ lies in the convex hull of always
four consecutive
control points $\Delta\bar p^\ell_i,\dots,\Delta\bar p^\ell_{i+3}$.
Now, we are going to prove
\begin{align}
\label{eq:lbnd}
	C_0 := 
  \liminf_{\ell \to \infty}
  \min_{i} \|\Delta\bar p^\ell_i\| &> 0 \\
\label{eq:ubnd}
	\lim_{\ell \to \infty}
  \max_{i} \|\Delta^2\bar p^\ell_i\| &= 0
	.
\end{align}
This means that, for $\ell$ sufficiently large, 
the control points are bounded away from the
origin, while their differences tend to zero, implying that
the convex hull of any $4$ consecutive control points does not
intersect the ball around the origin with radius $C_0/2$. 
Hence, neither the curves $\Phi^\ell[\Delta\bbP^\ell,\psi^\tau_1]$
nor their limit $\partial\Phi[\bP]$
intersect the interior of that ball, verifying the claim.

First, let $\bar c_1 := 4 \sqrt{n} c_1$. 
By \eqref{eq:D1}, \eqref{eq:Pi}, and $|\Pi \bq|_1 = |\be|_1 = 1$,
\[
  2|\Pi \bgl(\bq)|_1 \ge 2|\Pi \bgl(\be)|_1 - 4|\Pi (\bgl(\bq)-\bgl(\be))|_0
	\ge
	1 - 4 \sqrt{n} c_1 |\bd|_2
	=
	|\Pi \bq|_1(1 - \bar c_1 \kappa(\bq))
\]
for $\bq \in \Q^n[\delta_1]$.
Invariance under similarities yields
\begin{equation}
\label{eq:lowerbnd}  
  2\,|\Pi \bgl(\bp)|_1 \ge |\Pi \bp|_1 (1- \bar c_1\kappa(\bp))
	,\quad
	\bp \in \E^n[\delta_1]
	.
\end{equation}
Second, 
let $\ell_0 \in \N$ be chosen such that $\kappa_\ell(\bP) \le \min\{1/(2m), 1/(2 \bar c_1), \delta_1\}$
for all $\ell \ge \ell_0$.
We set $C := \min_i |\Pi T^n_i \bbP^{\ell_0}|_1$ and iterate \eqref{eq:lowerbnd}
to find 
\begin{equation}
\label{eq:lowerbnd1}
  \min_i |\Pi T^n_i \bbP^\ell|_1 \ge 
	C \prod_{j=\ell_0}^{\ell-1} (1 - \bar c_1 \kappa_j(\bP))
	\ge
	C \prod_{j=\ell_0}^{\infty} (1 - \bar c_1 \kappa_j(\bP))
  =: C'  
\end{equation}
for $\ell > \ell_0$, where $C' > 0$ because the sequence $\kappa_j(\bP)$
is essentially summable. 

Third, we note that 
\[
  \min_{0 \le j < n-1} \|\Delta q_j\| \ge 1 - |\bd|_1 \ge 1-m \kappa(\bq) \ge 1/2 = |\Pi \bq|_1/2
	,\quad
	\bq \in \Q^n[1/(2m)]
	.
\]
By invariance under similarities, 
\[
  \min_{0 \le j < n-1} ||\Delta p_j|| \ge |\Pi \bp|_1 /2
	,\quad
	\bp \in \E^n[1/(2m)]
	.
\]
Applying this estimate to \eqref{eq:lowerbnd1}, we obtain
\[
  \min_i \|\Delta \bar p_i^\ell \| \ge \min_i |\Pi T^n_i \bbP_\ell|_1 /2 \ge C'/2
	,\quad
	\ell > \ell_0
	,
\]
verifying \eqref{eq:lbnd}. 
Fourth, an analogous analysis to the above inequalities yields
\[
  2\,|\Pi \bgl(\bp)|_1 \le |\Pi \bp|_1 (1+ \bar c_1\kappa(\bp))
	,\quad
	\bp \in \E^n[\delta_1]
	.
\]
With $\ell_1 \in \N$ such that $\kappa_\ell(\bP) \le \delta_1$ for all $\ell \ge \ell_1$, we have
\[
  \max_i |\Pi T^n_i \bbP^\ell|_1 \le 
	\tilde{C} \prod_{j=\ell_1}^{\infty} (1 + \bar c_1 \kappa_j(\bP))
  =: \tilde{C}'
  ,
\]
where $\tilde{C} := \max_i |\Pi T^n_i \bbP^{\ell_1}|_1$ and $\tilde{C}' < \infty$ because the sequence $\kappa_j(\bP)$ is essentially summable.
Hence,
\[
  \kappa_\ell(\bP) = \max_i \frac{|T^n_i \bP^\ell|_2}{|\Pi T^n_i \bP^\ell|_1}
	= \max_i \frac{|T^n_i \bbP^\ell|_2}{|\Pi T^n_i \bbP^\ell|_1}
	\ge
	\frac{\max_i \|\Delta^2 \bar p^\ell_i\|}{\tilde{C}'}
	.
\]
Since $\kappa_\ell(\bP)$ is a null sequence, \eqref{eq:ubnd} follows
and the proof is complete.
\end{proof}


\section{Asymptotic analysis}
\label{sec:asymptotic}
In this section, we will relate higher order regularity properties of GLUE-schemes to
the derivatives $M_\lambda = D\bgl(\be), \lambda \in \{0,1\}$,
of $\bgl$ at $\be$ using the concept of the joint spectral radius. 
In general, the derivative $D\bgl$ of a self-map $\bgl :\E^n \to \E^n$
at the point $\bp \in \E^n$
is given by a set of $n \times n$ matrices $L_\lambda^{i,j}(\bp)$, each
of dimension $d \times d$,
acting on $\bp' =[p_0';\dots;p_{n-1}']\in \E^n$ according to
\[
  D\bgl(\bp) \cdot \bp' = 
  \begin{bmatrix}
  \sum_{j=0}^{n-1} p'_j L_\lambda^{0,j}(\bp) \\
	\vdots\\
  \sum_{j=0}^{n-1} p'_j L_\lambda^{n-1,j}(\bp) 
  \end{bmatrix}
  .
\]
If $\bgl$ commutes with similarities, the matrices $L_\lambda^{i,j}(\bp)$
have a special form.
Let $S = (\varrho,Q,s) \in \calS(\E)$ be any similarity according to
the specifications of the second section.
The invariance property (G) implies 
$D\bgl(S(\bp)) \cdot \bq Q = (D\bgl(\bp) \cdot \bq)Q$, and hence
\[
  Q\, L^{i,j}_\lambda(S(\bp)) = L^{i,j}_\lambda(\bp)\, Q
  ,\quad
  i,j = 0,\dots,n-1
  ,
  \lambda \in \{0,1\}
  .
\]
For $S = (\varrho,\id,s)$, we find
\(
\label{eq:Lij-1}
  L^{i,j}_\lambda(\varrho\bp+s) = L^{i,j}_\lambda(\bp)
  .
\)
That is, the derivative does not change when scaling or shifting
the argument. Further, $S = (1,Q,0)$ yields
\(
\label{eq:Lij-2}
  L^{i,j}_\lambda(\bp Q) = Q^{\rm t}\, L^{i,j}_\lambda(\bp)\, Q
  .
\)
By means of the last two displays, the derivative of $\bgl$ on $\LL^n_*$ is
completely determined by the derivative $M_\lambda$ at $\be$.
Now, we have to distinguish two cases:
\begin{itemize}
\item 
In the real-valued case $d=1$, the matrices $L^{i,j}$ are just scalars. Defining
the $(n \times n)$-matrices $A_0,A_1$ by $(A_\lambda)_{i,j} := L_\lambda^{i,j}(\be)$,
we obtain $M_\lambda \cdot \bp' = A_\lambda \bp'$. That is, the derivatives
at $\be$ are simply given by that pair of matrices. The real-valued case 
is excluded for the time being and
will be covered again by our considerations after the proof of Theorem~\ref{thm:C1AB}.
\item 
In the vector-valued case $d \ge 2$, things are more complicated.
This situation will be discussed now.
\end{itemize}
If $d \ge 2$, the structure of matrices $L_\lambda^{i,j}(\be)$ is
narrowed down as follows:
First,  
for $r = 2,\dots,d$, let $Q_r$ be the reflection changing the sign 
of the $r$th coordinate. Then $\be Q_r = \be$, and
the equalities
$L^{i,j}_\lambda(\be) = Q_r^{\rm t}\, L^{i,j}_\lambda(\be)\, Q_r$
imply that all $L^{i,j}_\lambda(\be)$ are diagonal matrices.
Second, for $r,s = 2,\dots,d$, let $Q_{r,s}$ be the reflection 
swapping the $r$th and $s$th coordinate.
Again, $\be Q_{r,s} = \be$, and
the equalities
$L^{i,j}_\lambda(\be) = Q_{r,s}^{\rm t}\, L^{i,j}_\lambda(\be)\, Q_{r,s}$
imply that all but the first entry on the diagonal of $L^{i,j}_\lambda(\be)$ 
coincide. That is, there exist real numbers $a^{i,j}_\lambda,b^{i,j}_\lambda$
such that
\[
  L^{i,j}_\lambda(\be) = 
  \begin{bmatrix}
  a^{i,j}_\lambda & 0 & \cdots & 0 \\
  0 & b^{i,j}_\lambda & \cdots & 0 \\
  \vdots &  \vdots    & \ddots &\vdots  \\
  0 & 0               & \cdots & b^{i,j}_\lambda
  \end{bmatrix}
  ,\quad
  \lambda \in \{0,1\}
  .
\]
Partitioning the $(n \times d)$-matrix $\bd$ into its first and the remaining columns,
\[
  \bd = [\bd_1, \bd_2]
  ,\quad
  \bd_1 \in \R^{n,1}
  ,\ 
  \bd_2 \in \R^{n,d-1}
  ,
\]
the image of $\bd$ under $M_\lambda$ can be written as
$M_\lambda \cdot \bd = [A_\lambda \bd_1, B_\lambda \bd_2]$,
where the coefficients of the $(n\times n)$-matrices $A_\lambda,B_\lambda$
are given by $a^{i,j}_\lambda,b^{i,j}_\lambda$, respectively.
Equally, 
\[
  M_\Lambda \cdot \bd := D\bgL(\be) = [A_\Lambda \bd_1, B_\Lambda \bd_2]
	,\quad
	\Lambda \in \calL 
	.
\]
The linear subdivision schemes 
corresponding to the pairs 
$(A_0,A_1)$ and $(B_0,B_1)$ of matrices 
are denoted by $\bA$ and $\bB$, respectively.
The following theorem shows that, in some sense, $\bG$
is at least as regular as the worse of $\bA,\bB$.
\begin{theorem}
\label{thm:C1AB}
If the chain $\bP$ is straightened by the GLUE-scheme $\bG$,
and if both associated linear schemes $\bA,\bB$ are almost $C^{1,\alpha}$,
then $\bG$ is almost $C^{1,\alpha}$ at $\bP$.
\end{theorem}
\begin{proof}
Consider any H{\"o}lder exponent $\alpha' < \alpha$. 
First, we show that there exists $\ell \in \N$
such that
\(
\label{eq:step1}
  -\ell^{-1}\log_2 \Gamma_\ell^*[0] \ge (\alpha + \alpha')/2
\)
with $\Gamma_\ell^*$ as defined in Lemma~\ref{lem:GammaBnd}.
To compute $\Gamma_\ell^*[0]$, we note that \eqref{eq:scale_bgl} implies
$|\Pi\bgL(\be)|_1 = 2^{-\ell}$.
Further,
by lemmas~\ref{lem:GammaBnd} and \ref{lem:g(e)}, 
\[
  \bigl|M_\Lambda\bigr|_2
	=
	\max_{\bd \neq 0} 
	\frac{\bigl|[\Delta^2 A_\Lambda \bd_1, \Delta^2 B_\Lambda \bd_2]\bigr|_0}{|\bd|_2}
	=
	\max_{\bd \neq 0} 
	\frac{\bigl|[A_{2,\Lambda}\Delta^2\bd_1, B_{2,\Lambda} \Delta^2\bd_2]\bigr|_0}{|\bd|_2}
	,
\]
where $A_{2,\Lambda},B_{2,\Lambda}$ are products of matrices
corresponding to the derived schemes $\bA_2,\bB_2$, respectively.
Recalling  \eqref{eq:rho}, 
we define $\mu_\ell := \max\{\varrho_\ell(\bA_2),\varrho_\ell(\bB_2)\}$
and $\mu := \max\{\varrho(\bA_2),\varrho(\bB_2)\}$.
The entries of the $(n \times d)$-matrix in the numerator above cannot exceed
$\mu^\ell_\ell |\Delta^2\bd|_0 = \mu^\ell_\ell|\bd|_2$. 
Hence, $\bigl|M_\Lambda\bigr|_2 \le \mu_\ell^\ell \sqrt{d}$, implying
\[
  \Gamma_\ell^*[0] =
	\max_{\Lambda \in \calL^\ell} \frac{|M_\Lambda|_2}{|\Pi\bgL(\be)|_1}
	\le
	(2 \mu_\ell)^\ell \sqrt{d}
	.
\]
By assumption, both $\bA$ and $\bB$ are almost $C^{1,\alpha}$ so that,
by \eqref{eq:jsr}, $\log_2 2\mu \le -\alpha$.
Therefore,
\[
  \liminf_{\ell \to \infty} (-\ell^{-1}\log_2 \Gamma_\ell^*[0]) \ge
	-\limsup_{\ell\to \infty} (\log_2 2\mu_\ell + \ell^{-1} \log_2\sqrt{d})
	= -\log_2 2\mu \ge \alpha
	,
\]
showing that \eqref{eq:step1} holds true when choosing $\ell \in \N$ sufficiently large .

Second, we fix $\ell$ as found above. The expression
$-\ell^{-1} \log_2 \Gamma_\ell^*[\delta]$ is a continuous function of $\delta$ in a neighborhood
of $\delta=0$. Hence, choosing $\delta>0$ small enough, the function values
$-\ell^{-1} \log_2 \Gamma_\ell^*[\delta]$ and $-\ell^{-1} \log_2 \Gamma_\ell^*[0]$ differ by less
than $(\alpha-\alpha' )/2$, and we obtain
\[
  -\ell^{-1} \log_2 \Gamma_\ell^*[\delta] \ge -\ell^{-1} \log_2 \Gamma_\ell^*[0] - (\alpha-\alpha')/2
	\ge \alpha'
	.
\]
Third, let $\bP$ be a chain that is straightened by $\bG$. Then there exists
$s \in \N$ such that $\kappa_s(\bP) \le \delta$, i.e., $\bP^s \in \E^{\ge n}[\delta]$.
By Lemma~\ref{lem:straight}, $\bP^s$ is straightened by $\bG$ at rate $\alpha'$, and so is $\bP$.
With Theorem~\ref{thm:C1alpha}, we have that $\bG$ is $C^{1,\alpha'}$ at $\bP$.
\end{proof}
As we have seen, the
vector-valued case $d \ge 2$ leads to two linear schemes $\bA$ 
and $\bB$ corresponding, in some sense, to the tangential and the normal component
of the limit curve. In the real-valued case $d=1$, there is only a single
scheme $\bA$, given by a
pair of matrices $A_0,A_1$ representing the derivatives of $\bg_0,\bg_1$
at $\be$. This is possible also for arbitrary space dimension $d$ 
if the schemes $\bA$ and $\bB$ coincide.
Thus,
we terminate the special treatment of the vector-valued case $d \ge 2$,
return to the general setting $d \in \N$ and elaborate on the following
special case:
\begin{definition}
\label{def:quasi}
A GLUE-scheme $\bG$ is called {\em locally linear}
if there exist $(n \times n)$-matrices
$A_0,A_1$ such that
\[
  M_\lambda\cdot \bp' = A_\lambda \bp'
  ,\quad
  \bp' \in \E^n
  ,\
  \lambda \in \{0,1\}
  .
\] 
The linear subdivision scheme $\bA$ corresponding to the matrices $A_0,A_1$
is called the {\em linear companion} of $\bG$.
\end{definition}

For instance, as shown in \cite{BspSabin05}, CPS is locally linear,
and the four-point scheme is its linear companion. Trivially, real-valued schemes
are always locally linear. Since the case $d=1$ was excluded in
Theorem~\ref{thm:C1AB}, we state for the sake of completeness:
\begin{corollary}
If the chain $\bP$ is straightened by the locally linear GLUE-scheme $\bG$
and if its linear companion $\bA$ is almost $C^{1,\alpha}$,
then $\bG$ is almost $C^{1,\alpha}$ at $\bP$.
\end{corollary}
We skip the pending proof for the case 
$d=1$, which follows exactly the ideas used to establish Theorem~\ref{thm:C1AB}.

Equations \eqref{eq:Lij-1} and \eqref{eq:Lij-2} imply that
locally linear schemes have constant derivative on the space
of non-constant linear subchains,
\[
  D\bgl(\bp)\cdot \bp' = A_\lambda \bp'
  ,\quad
  \bp \in \LL^n_*
  ,\
  \bp' \in \E^n
  .
\]
Moreover, 
\[
  (1+\varepsilon)\bgl(\bp) =
  \bgl(\bp + \varepsilon\bp) = 
  \bgl(\bp) + \varepsilon A_\lambda \bp
  +O(\varepsilon^2)
  ,\quad
  \bp \in \LL^n_*
  ,
\]
shows that $\bgl(\bp) = A_\lambda \bp$. Hence,
the schemes $\bG$ and $\bA$ coincide on the space
of linear chains,
\(
\label{eq:G=A}
  \bG(\bP) = \bA \bP
  ,\quad
  |\bP|_2 = 0
  .
\)
For non-linear chains, the deviation is bounded as follows:
\begin{lemma}
\label{lem:RQLin}
Let $\bA$ be the linear companion of the locally linear
GLUE-scheme $\bG$. The remainder
\[
  \bR(\bP) := \bG(\bP)-\bA \bP
\]
is bounded by
\[
  |\bR(\bP)|_0 \le c_1 \kappa(\bP)^\nu \, |\bP|_2
  ,\quad
  \bP \in \E^{\ge n}[\delta_1]
  ,
\]
with $c_1,\delta_1$ as in Lemma~\ref{lem:g(e)}, and $\nu$ the
regularity parameter of $\bG$. 
\end{lemma}
\begin{proof}
For $\bp \in T^n\bP$, let $S_\bp$ be the similarity as introduced
in Section~\ref{sec:setup}, i.e., $S_\bp(\bp) = \bq = \be + \bd \in \Q^n$.
Then, by \eqref{eq:D2},\eqref{eq:G=A}, and \eqref{eq:kappa=d},
\[
	|S_\bp(\bgl(\bp)-A_\lambda\bp)|_0 
	=
	|\bgl(\be+\bd)-\bgl(\be) - A_\lambda \bd|_0
  \le
  c_1 |\bd|_2^{1+\nu}
	=
	c_1 \kappa(\bp)^\nu |S_\bp(\bp)|_2
	.
\]
Dividing this estimate by $|S_\bp|$ yields
$|\bgl(\bp)-A_\lambda\bp|_0\le c_1 \kappa(\bp)^\nu |\bp|_2$.
\end{proof}
Now, we are prepared to present our main result concerning
regularity of locally linear schemes. It states that, essentially,
second order H{\"o}lder regularity is passed on from $\bA$ to $\bG$
if the regularity parameter of $\bG$ is sufficiently large.
\begin{theorem}
Let $\bG$ be a locally linear
GLUE-scheme with linear companion $\bA$, spread $n$, and regularity parameter $\nu$.
If the chain $\bP \in \E^{\ge n}$ is straightened by $\bG$
and if $\bA$ is almost $C^{2,\alpha}$,
then $\bG$ is almost $C^{2,\beta}$ at $\bP$,
where $\beta := \min\{\alpha,\nu\}$.
\end{theorem}
\begin{proof}
Given any $\gamma\in(0,\beta)$, let $\mu := (\beta+\gamma)/2$.
Further, we abbreviate
$\bar\gamma := 2+\gamma, \bar\mu := 2+\mu$.
First, we derive bounds on $\bR(\bP^\ell)$ and $\Delta^3\bP^\ell$. 
Being almost $C^{2,\alpha}$, the scheme $\bA$ is also almost $C^{1,1}$.
Then we know from the proof of Theorem~\ref{thm:C1AB} that $\bP$
is straightened at rate $\alpha':= (1+\mu)/(1+\nu)<1$. 
By lemmas~\ref{lem:RQLin} and \ref{lem:halvening},
\(
\label{eq:R}
	|\bR(\bP^\ell)|_0 \le
	c_1 \kappa(\bP^\ell)^\nu |\bP^\ell|_2
	\le
	C 2^{-\ell(\alpha' \nu + 1 + \alpha')}
	= C 2^{-\ell\bar\mu}
\)
whenever $\bP^\ell \in \E^{\ge n}[\delta_1]$. However, since $\bP$
is straightened by $\bG$, there exists $\ell_0 \in \N$ such that
this is true for all $\ell \ge \ell_0$.
Fixing that $\ell_0$, we re-write \eqref{eq:G-A} in the form 
\[
  \bP^\ell = \bA^{\ell-\ell_0}\bP^{\ell_0} + \sum_{i=\ell_0}^{\ell-1} \bA^{\ell-i-1}\bR(\bP^i)
\]
and obtain 
\[
  |\bP^{\ell}|_3 \le
	|\bA^{\ell-\ell_0} \bP^{\ell_0}|_3 + \sum_{i=\ell_0}^{\ell-1} |\bA^{\ell-i-1} \bR(\bP^i)|_3 
	\le
	|\bA^{\ell-\ell_0}_3 \Delta^3\bP^{\ell_0}|_0 
	+ \sum_{i=\ell_0}^{\ell-1} |\bA_3^{\ell-i-1} \Delta^3\bR(\bP^i)|_0
\]
for all $\ell \ge \ell_0$.
Since $\bar\gamma < 2+\alpha$ and $\varrho(\bA_3) \le 2^{-(2+\alpha)}$, 
there exist constants $c$ and $c' = c/8$ such that
\[
	|\bA_3^r \Delta^3\bP'|_0 \le c' 2^{-r\bar\gamma} |\Delta^3 \bP'|_0
	\le c 2^{-r\bar\gamma} |\bP'|_0
	,\quad
	r \in \N_0
	,\
	\bP' \in \E^{\ge n}
	.
\]
Together, the last three displays yield
\begin{align*}
  |\bP^\ell|_3 &\le 
	c \Bigl(2^{(\ell_0-\ell)\bar\gamma}|\bP^{\ell_0}|_0 + 
	C \sum_{i=\ell_0}^{\ell-1} 2^{(i+1-\ell)\bar\gamma} 2^{-i\bar\mu}
	\Bigr)\\
	& = 
	c 2^{-\ell\bar\gamma} \Bigl(
	2^{\ell_0\bar\gamma}|\bP^{\ell_0}|_0 
	+ C 2^{\bar\gamma} \sum_{i=\ell_0}^{\ell-1} 2^{i(\bar\gamma-\bar\mu)}
	\Bigr)
	.
\end{align*}
Since $\bar\gamma<\bar\mu$, the sum $\sum_i$ is bounded independent of $\ell$. Hence, there
exists a constant $C'$ with
\(
\label{eq:D3P}
|\bP^\ell|_3 \le C' 2^{-\ell\bar\gamma}
	,\quad
	\ell \ge \ell_0
	.
\)
Second, we show that $\bG$ is $C^2$ at $\bP$.
By Lemma~\ref{lem:proxy} and \eqref{eq:R}, 
\(
\label{eq:D2Phi}
  \bigl|\partial^2(\Phi^{\ell+r}[\bP^{\ell+r},\psi] -
	\Phi^{\ell}[\bP^{\ell},\psi])\bigr|_z \le
  c\, \sum_{i=\ell}^\infty 2^{2i}|\bR(\bP^i)|_0
	\le
	c C \, \sum_{i=\ell}^\infty 2^{-i \mu}
	\le
	C' 2^{-\ell\mu}
  .
\)
Hence, $\partial^2\Phi^{\ell}[\bP^\ell,\psi]$ is a Cauchy sequence with limit $\partial^2\Phi[\bP]$.

Third, we determine the local H{\"o}lder regularity of $\partial^2\Phi[\bP]$.
Given $h \in (0,h_0]$, choose $\ell\in \N$ such that
$2^{-\ell} < h \le 2^{-\ell+1}$. If $h$ is small enough, 
$\ell$ is sufficiently large to fulfill the
estimates above.
The modulus of continuity of the second derivative of the limit curve is bounded by
\[
  \omega_z(\partial^2\Phi[\bP],h) \le 
  \omega_z(\partial^2(\Phi[\bP]-\Phi^\ell[\bP^\ell,\psi]),h)
  +
  \omega_z(\partial^2\Phi^\ell[\bP^\ell,\psi],h)
	=: s_1+s_2
  .
\]
By \eqref{eq:D2Phi},
the first summand $s_1$ is bounded by
\begin{align*}
  s_1
  \le
  2\, \bigl| \partial^2(\Phi[\bP]-\Phi^\ell[\bP^\ell,\psi]) \bigr|_z 
	\le
  2 C' 2^{-\ell\mu} 
	\le 
	2 C' h^\gamma
  .
\end{align*}
For the second one, we obtain by \eqref{eq:diff}
\[
  s_2 = \omega_z(\Phi^\ell[2^{2\ell} \Delta^2\bP^\ell,\psi_2],h)
	\le \sup_{2^{-\ell}<h\le 2^{-\ell+1}} 2^{2\ell} \max_{\tau\in I_{z+h/2}}
	\|\sigma(\tau,h)\|
	,
\]
where
\[
  \sigma(\tau,h) := 
	\sum_{i = 0}^{N^\ell-1} \bigl(\psi_2(2^\ell\tau_h^+ -i) - \psi_2(2^\ell\tau_h^--i) \bigr)
	\Delta^2 p_i^\ell
  ,\quad
  \tau_h^\pm := \tau \pm h/2
	.
\]
By \eqref{eq:partition}, the function $\psi_2$ constitutes a partition of unity.
For levels
$\ell \ge \ell_z(\psi_2)$, see Section~\ref{sec:regular} for the definition,
we obtain by~\eqref{eq:sum1}
\[
  \sigma(\tau,h) =
	\sum_{i = 0}^{N^\ell-1} \bigl(\psi_2(2^\ell\tau_h^+ -i) - \psi_2(2^\ell\tau_h^- -i) \bigr)
	(\Delta^2 p_i^\ell - \Delta^2 p_j^\ell)
  .
\]
The index $j$ is defined as the integer closest to $2^\ell \tau$, i.e.,
$|j-2^\ell \tau|\le 1/2$.
Then, with $r$ such that $\operatorname{supp} \psi_2 \subseteq [-r+3/2,r-3/2]$, 
we find $\psi_2(2^\ell\tau_h^\pm -i) = 0$ whenever $|i-j| > r$.
Hence,
\[
  \sigma(\tau,h) =
	\sum_{|i-j|\le r}  \bigl(\psi_2(2^\ell\tau_h^+ -i) - \psi_2(2^\ell\tau_h^- -i) \bigr)
	(\Delta^2 p_i^\ell - \Delta^2 p_j^\ell)
	.
\]
This expression can be estimated by \eqref{eq:Lebesgue} using the
bound $\|\Delta^2 p_i^\ell - \Delta^2 p_j^\ell\| \le r |\bP^\ell|_3$
and \eqref{eq:D3P},
\[
  \|\sigma(\tau,h)\| \le 2\bar\psi_2 r |\bP^\ell|_3
	\le 
	2\bar\psi_2 r C' 2^{-\ell\bar\gamma}
	= C''2^{-\ell\bar\gamma}
	.
\]
Hence, $s_2 \le C'' 2^{-\ell\gamma} \le C'' h^\gamma$, 
showing that $\omega_z(\partial^2\Phi[\bP],h) \le C h^\gamma$, as requested.
\end{proof}
Unfortunately, regularity of $\bA$ beyond $C^{2,1}$ is not
necessarily
inherited by $\bG$. As an example, consider the scheme
$\bG$, given by
\begin{align*}
  g_0(T_i^3\bP) &= 
  (5 p_{i-1} + 10 p_i + p_{i+1})/16 +
  \frac{\|\Delta^2 p_{i-1}\|}{\|p_{i+1}- p_{i-1}\|}\, \Delta^2 p_{i-1}\\
  g_1(T_i^3\bP) &= (p_{i-1} + 10 p_i + 5 p_{i+1})/16
\end{align*}
Its linear companion is the quartic B-spline scheme $\bA^{1/2}$, as
discussed in Section~\ref{sec:regular}, which is $C^{3,1}$, 
while numerical experiments
suggest that $\bG$ is not smoother than $C^{2,1}$.

\section{Conclusion}

We have presented a general framework for the analysis of geometric subdivision schemes.
It is related to standard linear $C^{1,\alpha}$-theory by considering the decay 
not of second differences, but of another
quantity (relative distortion) measuring the deviation from a linear behavior.
It is also related to known approaches for the analysis of manifold-valued
subdivision as it uses a proximity condition. However, for $C^{1,\alpha}$-analysis,
the linear reference scheme depends only on the shift of the given scheme, and is
indifferent otherwise. Unlike any other known analysis of non-linear schemes, 
our approach admits to establish H{\"o}lder regularity by a universal procedure,
which can be fully automated. This procedure consists of two steps: First,
an upper bound on the maximal relative distortion is determined together
with some parameter characterizing the H{\"o}lder exponent. This computation
can be done once and for all for a given scheme. Second, when a specific initial
chain is given, this chain is checked for compliance with the pre-computed bound.
If this bound is met, convergence and regularity of the limit are verified.
Otherwise, a few rounds of subdivision are applied to the initial data,
and then the check is repeated. If the check fails even after many subdivision steps,
then it is conceivable that the limit does not have the expected regularity for this
specific set of initial data\footnote{It should be noted that the regularity of the limit curve
may indeed depend on the initial data.
For instance, this phenomenon can be observed for median-interpolating subdivision \cite{NonlinYu05}, 
where non-monotonic data may yield non-differentiable limits.}.
In a forthcoming report, we will describe details of an implementation.

There are many obvious directions for future research. For instance, 
higher order H{\"o}lder regularity could be approached by defining appropriate proximity conditions
or by considering deviations from a circle instead of a straight line.
Further, we conjecture that Theorem~\ref{thm:C1AB} could be enhanced by showing that 
regularity is determined only by the normal scheme $\bB$, while the tangential scheme $\bA$
is irrelevant. When trying to further generalize the class of schemes, it would be
nice to drop property (E), but this seems to be a real challenge as it is not clear
at all how to choose an appropriate parametrization, or to get rid of a specific
parametrization at all and come up with a genuine geometric proof.

\bibliographystyle{alpha}
\bibliography{references}

\end{document}